\documentclass[a4paper,11pt]{amsart}
\usepackage{amssymb}
\usepackage[utf8]{inputenc}
\usepackage{verbatim}

\usepackage{xspace}
\usepackage{enumitem}

\newlist{enumup}{enumerate}{1}
\newlist{enumequi}{enumerate}{1}
\newlist{enumprop}{enumerate}{1}
\setlist[enumup,1]{label=\textup{(\alph{*})}, ref=\textup{(\alph{*})}}
\setlist[enumequi,1]{label=\textup{(\roman{*})}, ref=\textup{(\roman{*})}}
\setlist[enumprop,1]{label=\textup{(\arabic{*})}, ref=\textup{(\arabic{*})}}
\setlist[enumerate,1]{label=\textup{(\alph{*})}, ref=\textup{(\alph{*})}}
\setlist[enumerate,2]{label=\textup{(\Alph{*})}, ref=\textup{(\Alph{*})}}

\DeclareMathOperator{\diam}{diam}
\newcommand{\eps}{\varepsilon}
\newcommand{\spann}{\mbox{span}}

\newcommand{\HB}{\text{H{\kern -0.35em}B}}
\newcommand{\vertiii}[1]{{\left\vert\kern-0.25ex\left\vert\kern-0.25ex\left\vert#1\right\vert\kern-0.25ex\right\vert\kern-0.25ex\right\vert}}

\newcommand{\CWO}{\textit{CWO}\xspace}
\newcommand{\CWOSX}{\textit{CWO}-$S$\xspace}
\newcommand{\CWOBX}{\textit{CWO}-$B$\xspace}
\newcommand{\CCOMtwo}{\textup{(\textit{co}2)}\xspace}
\newcommand{\CCOMn}{\textup{(\textit{co\hskip1pt n})}\xspace}
\newcommand{\CCOM}{\textup{(\textit{co})}\xspace}

\newcommand{\K}{\mathbb{K}}
\newcommand{\C}{\mathbb{C}}
\newcommand{\N}{\mathbb{N}}

\newcommand{\R}{\mathbb{R}}
\newcommand{\D}{\mathbb{D}}
\newcommand{\T}{\mathbb{T}}
\newcommand{\Xs}{X^\ast}
\newcommand{\xs}{x^\ast}
\newcommand{\zs}{z^\ast}
\DeclareMathOperator{\supp}{supp}
\DeclareMathOperator{\re}{Re}

\DeclareMathOperator{\ext}{ext}
\DeclareMathOperator{\conv}{conv}
\DeclareMathOperator{\clco}{\overline{conv}}
\newcommand{\pten}{\ensuremath{\hat{\otimes}_\pi}}
\newcommand{\iten}{\ensuremath{\hat{\otimes}_\varepsilon}}

\newtheorem{thm}{Theorem}[section]
\newtheorem{prop}[thm]{Proposition}
\newtheorem{lem}[thm]{Lemma}
\newtheorem{cor}[thm]{Corollary}
\newtheorem*{Claim}{\sc Claim}

\theoremstyle{definition}
\newtheorem{defn}[thm]{Definition}

\newtheorem{example}[thm]{Example}

\theoremstyle{remark}

\newtheorem{rem}[thm]{Remark}

\numberwithin{equation}{section}

\author[Abrahamsen et al.]{Trond Arnold Abrahamsen \and Julio Becerra
  Guerrero \and Rainis Haller \and Vegard Lima \and M\"{a}rt
  P\~{o}ldvere}

\address[T.~A.~Abrahamsen]{Department of Mathematics, University of
  Agder, Postbox 422, 4604 Kristiansand, Norway.}
\email{trond.a.abrahamsen@uia.no}
\urladdr{http://home.uia.no/trondaa/index.php3}

\address[J.~B.~Guerrero]{Universidad de Granada, Facultad de Ciencias,
Departemento de An\'{a}lisis Mathem\'{a}tico, 18071-Granada, Spain}
\email{juliobgn@ugr.es}

\address[R.~Haller]{University of Tartu, J. Liivi 2,
50409 Tartu, Estonia.}
\email{rainis.haller@ut.ee}

\address[V.~Lima]{Department of Engineering Sciences, University of
  Agder, Postbox 422, 4604 Kristiansand, Norway.}
\email{vegard.lima@uia.no}

\address[M.~P\~{o}ldvere]{University of Tartu, J. Liivi 2,
50409 Tartu, Estonia.}
\email{mart.poldvere@ut.ee}

\thanks{The second named author was partially supported by
MEC (Spain) Grant MTM2014-58984-P
and Junta de Andaluc\'ia grants FQM-0199, FQM-1215.}
\thanks{R. Haller and M. P\~{o}ldvere were partially supported
by institutional research funding IUT20-57
of the Estonian Ministry of Education and Research.}

\title[CWO Banach spaces]{Banach spaces where convex combinations
 of relatively weakly open subsets of the unit ball
 are relatively weakly open}

\begin{document}

\subjclass[2010]{46B04, 46B20}
\keywords {relatively weakly open set, convex combinations of
  relatively weakly open sets, scattered locally compact, locally
  complemented}

\begin{abstract}
  We introduce and study Banach spaces which have property \CWO,
  i.e., every finite convex combination of relatively
  weakly open subsets of their unit
  ball is open in the relative weak topology of the unit
  ball. Stability results of such spaces are established, and we
  introduce and discuss a geometric condition---property \CCOM---on a
  Banach space. Property \CCOM essentially says
  that the operation of taking convex combinations
  of elements of the unit ball is, in a sense, an open map.
  We show that if a finite dimensional Banach space
  $X$ has property \CCOM, then for any scattered locally compact Hausdorff
  space $K$, the space $C_0(K,X)$ of continuous $X$-valued functions vanishing at
  infinity has property \CWO. Several Banach spaces are proved
  to possess this geometric property; among others: 2-dimensional real
  spaces, finite dimensional strictly convex spaces,
  finite dimensional polyhedral spaces,
  and the complex space $\ell_1^n$.
  In contrast to this, we provide an example of a $3$-dimensional
  real Banach space $X$
  for which $C_0(K,X)$ fails to have property \CWO.

  We also show that $c_0$-sums of finite dimensional Banach spaces
  with property \CCOM have property \CWO. In particular, this
  provides examples of such spaces outside the class of $C_0(K,X)$-spaces.
\end{abstract}
\maketitle

\section{Introduction}
\label{sec:intro}
In this paper we consider Banach spaces over the
scalar field $\K$, where $\K$ is either the
real field $\R$ or the complex field $\C$.
If not mentioned explicitly the space involved
could be either real or complex.
By $S_X, B_X$, and
$B_X^\circ$ we denote respectively the unit sphere, unit ball, and
open unit ball of a Banach space $X$.
The topological dual of $X$ is denoted by
$X^*$. By a slice (of the unit ball) we mean a set
of the form
\begin{align*}
  S(x^*, \eps) := \{x \in B_X: \mbox{Re}\,x^*(x) > 1 - \eps\},
\end{align*}
where $\eps > 0$ and $x^* \in S_{X^*}$.
A topological space $K$ is said to be \emph{scattered}
if every non-empty subset $A$ of $K$ contains a point which is
isolated in $A$.

Let $\mathcal{F} = \{ f \in L_1[0,1] : f \ge 0, \|f\| = 1\}$.
It was shown in \cite[Remark~IV.5, p.~48]{MR912637},
that $\mathcal{F}$ has ``a remarkable geometrical property'':
any convex combination of a finite number of relatively
weakly open subsets (in particular, slices) of $\mathcal{F}$
is still relatively weakly open.
Recently, in \cite{2017arXiv170106169A},
it was shown that if $K$ is a scattered
compact Hausdorff space, then the space $C(K,\K)$
of continuous $\K$-valued functions on $K$
has the property
that finite convex combinations of slices of $B_{C(K,\K)}$
are relatively weakly open in $B_{C(K,\K)}$.
Subsequently, in \cite{2017arXiv170302919H}, it was shown
that $C(K,\K)$ has this property if and only if
$K$ is scattered.
In fact, in \cite{2017arXiv170302919H} the result was proven
for the space $C_0(K,\K)$
of continuous $\K$-valued functions on $K$ vanishing at infinity,
where $K$ is a locally compact Hausdorff space.
The results in \cite{2017arXiv170106169A} are
true also in this setting.

The main focus of this paper is to prove that,
for some Banach spaces $X$,
the space $C_0(K,X)$ of $X$-valued continuous functions on
a scattered locally compact Hausdorff space $K$
also satisfies the property that finite convex combinations
of slices of $B_{C_0(K,X)}$ are relatively weakly open in
$B_{C_0(K,X)}$. We will prove this by showing that even finite convex
combinations of relatively weakly
open subsets of the unit ball $B_{C_0(K,X)}$ are relatively weakly
open in the unit ball.
More specifically, we consider the following properties.

\begin{defn}\label{defn:cwos}
Let $X$ be a Banach space. We say that
\begin{enumerate}
 \item $X$ has property \CWO, if,
   for every finite convex combination $C$ of relatively
   weakly open subsets of $B_X$, the set $C$ is open in the
   relative weak topology of $B_X$;
\item $X$ has property \CWOSX, if, for every finite convex
   combination $C$ of relatively weakly open subsets of
   $B_X$, every $x \in C\cap S_X$ is an interior point of $C$ in the
   relative weak topology of $B_X$;
\item $X$ has property \CWOBX, if, for every finite convex
  combination $C$
  of relatively weakly open subsets of $B_X$, every
  $x \in C\cap B_X^\circ$ is an interior point of $C$ in the relative weak
  topology of $B_X$.
\end{enumerate}
\end{defn}

It is clear that a Banach space $X$
has property \CWO if and only if it has
both properties \CWOSX and \CWOBX.
We will show in Theorem~\ref{thm:L1mu-CWOSX}
that $L_1[0,1]$ has property \CWOSX,
but it fails to have property \CWOBX
by Corollary~\ref{cor:L1_fails_CWOBXo}.
In fact, any $\ell_1$-sum of
two spaces with property \CWOSX has property \CWOSX,
but fails to have property \CWOBX
\cite[Theorem~2.3 and Proposition~2.1]{2017arXiv170302919H}.
Let $K$ be a scattered locally compact Hausdorff space.
In
Example~\ref{example: a space X without (P10) such that C_0(K,X) is not CWO1},
we give an example of a finite dimensional Banach space $X$
such that $C_0(K,X)$ has property \CWOBX,
by Theorem~\ref{thm: main results on CWO-s for C_0(K,X)},
but $C_0(K,X)$ fails property \CWOSX,
by Proposition~\ref{prop: a class of Banach spaces for which C_0(K,X) is never CWO1 (and X fails (P10))}.
Thus neither of the properties \CWOSX and \CWOBX implies the other.

Note that, in the definition of property \CWOSX, the intersection
$C \cap S_X$ may be empty. Every
strictly convex Banach space has property \CWOSX.
Indeed, let $C$ denote a finite convex combination of
relatively weakly open subsets $W_1,\dots,W_n$ of $B_X$.
Then every $x \in C \cap S_X$ is an extreme point of $B_X$,
hence $x \in \bigcap_{j=1}^n W_j$ which is a relatively
weakly open neighbourhood of $x$ contained in $C$.

We also remark that if a Banach space $X$ has
property \CWOBX, then, by \cite[Theorem~2.4]{2017arXiv170304749L},
every finite convex combination of slices of $B_X$
has diameter two, that is, $X$ has the
\emph{strong diameter two property}.
This is not the case for property \CWOSX
since, for example, $\ell_2$ has this property.

Let $K$ be a locally compact Hausdorff space.
It is known that $C_0(K,X)$ can be identified with the
injective tensor product $C_0(K) \iten X$.
It is also known that the (injective) tensor product
$X \iten Y$ of two Banach spaces $X$ and $Y$
contains one complemented isometric copies of
both $X$ and $Y$.
We will show in Proposition~\ref{prop:norm1compl-inherit-cwox}
that property \CWO is inherited by one complemented
subspaces.
Hence in order for $C_0(K,X)$ to have property \CWO,
it is necessary that both $C_0(K)$ and $X$
have property \CWO.
By \cite[Theorem~3.1]{2017arXiv170302919H},
this implies that $K$ must be scattered.
Hence we will only consider scattered locally compact
Hausdorff spaces $K$.

The paper is organized as follows.
Section \ref{sec:3} is about spaces of the type $C_0(K,X)$,
where $K$ is a scattered locally compact Hausdorff space
and $X$ a finite dimensional Banach space.
We establish and discuss here a geometric condition on
$X$, property \CCOM, guaranteeing that $C_0(K,X)$ has property \CWO;
see Theorem~\ref{thm: main results on CWO-s for C_0(K,X)}.
We prove that all strictly convex spaces have property \CCOM.
We also show that $C_0(K,X)$ has property \CWOBX
whenever $X$ is finite dimensional.

In Section \ref{sec:4} we show that any two dimensional real Banach space
has property \CCOM, but there exists a three dimensional
real Banach space which fails property \CWOSX.
This shows that property \CWO is strictly stronger
than property \CWOBX. We also show that if the dual of a finite dimensional
Banach space $X$ (real or complex) is polyhedral,
then $X$ has property \CCOM.
Finally we show that both the real and complex $\ell_1^n$ have property
\CCOM.
It should be noted in this connection that, in the complex case,
$(\ell_1^n)^* =\ell_\infty^n$ is not a polyhedral space
while $\ell_1^n$ is.

In Section \ref{sec:2} we prove,
in Proposition~\ref{prop:norm1compl-inherit-cwox},
that all the \CWO-properties are stable by taking one complemented
subspaces. We also show that if $X$ contains a complemented
subspace isomorphic to $\ell_1$, then $X$ does not
have property \CWOBX.

In Section \ref{sec:5} we show that $c_0$-sums of finite dimensional
Banach spaces with property \CCOM have property \CWO.
This result provides examples of spaces with property \CWO
outside the class of $C_0(K,X)$-spaces discussed in the two
previous sections. We end this section by showing that
the real space $L_1(\mu)$
has property \CWOSX provided $\mu$ is a non-zero $\sigma$-finite
(countably additive non-negative) measure.

We follow standard Banach space notation
as can be found, e.g., in the book \cite{MR2766381}.
As mentioned above, we consider Banach spaces
over the scalar field $\K$, where $\K = \R$ or
$\K = \C$.
We use the notation
$\T = \{ \alpha \in \K : |\alpha| = 1 \}$
and
$\D = \{ \alpha \in \K : |\alpha| \le 1 \}$.

\section{A geometric condition for Banach spaces $X$ guaranteeing
  that $C_0(K,X)$ has property \CWO}
\label{sec:3}
Our main objective in this section is to establish a geometric
condition for finite dimensional Banach spaces $X$ guaranteeing that
the space $C_0(K,X)$, where $K$ is a scattered locally compact
Hausdorff space, has property \CWO.

\begin{defn}\label{def:P11}
  Let $X$ be a Banach space.
  We say that a point $x\in B_X$ has \emph{property \CCOM},
  if for every $n \in \N$, $n \ge 2$,
  \begin{itemize}
    \item[\CCOMn]
      whenever
      $x_1,\dotsc,x_n \in B_X$ and $\lambda_1,\dotsc,\lambda_n > 0$,
      $\sum_{j=1}^n \lambda_j = 1$, are such that
      $x = \sum_{j=1}^n \lambda_j x_j$, and $\eps>0$,
      there is a $\delta>0$ such that,
      setting $B:=B(x,\delta)\cap B_X$,
      there are continuous functions
      \begin{equation*}
        \widehat{v}_j \colon \,B \to B_X,
        \qquad j \in \{1,\dotsc,n\},
      \end{equation*}
      such that, for every $u \in B$,
      \begin{equation}\label{eq:P11-conditions}
        u = \sum_{j=1}^n \lambda_j \widehat{v}_j(u)
        \qquad \text{and} \qquad \|\widehat{v}_j(u) - x_j\| < \eps,
        \quad j \in \{1,\dotsc,n\}.
      \end{equation}
  \end{itemize}

  We say that the space $X$ has \emph{property \CCOM}
  if every point $x\in B_X$ has property \CCOM.
\end{defn}

For finite dimensional Banach spaces, property \CCOM implies property \CWO.
We will see in Proposition~\ref{prop: a class of Banach spaces for
  which C_0(K,X) is never CWO1 (and X fails (P10))}
and Example~\ref{example: a space X without (P10) such that C_0(K,X)
  is not CWO1}
that not every finite dimensional Banach space has property \CWO.

\begin{prop}\label{prop:fin_dim_co_cwo}
  Let $X$ be a finite dimensional Banach space
  with property \CCOM.
  Then $X$ has property \CWO.
\end{prop}

\begin{proof}
  Let $n\in\N$, let  $U_1,\dotsc,U_n$ be relatively weakly open
  subsets of $B_X$, and
  let $\lambda_1,\dotsc,\lambda_n > 0$ with $\sum_{j=1}^n \lambda_j = 1$.
  If $x = \sum_{j=1}^n \lambda_j x_j$
  where $x_j \in U_j$, then
  there exists an $\varepsilon > 0$
  such that $B(x_j,\varepsilon) \cap B_X \subset U_j$
  for every $j \in\{1,\dotsc,n\}$.

  By assumption, we can find a $\delta > 0$ and functions
  $\widehat{v}_j : B(x,\delta) \cap B_X \to B_X$
  such that, for every $u \in B(x,\delta) \cap B_X$,
  we have $\widehat{v}_j(u) \in B(x_j,\varepsilon)$ and
  $\sum_{j=1}^n \lambda_j \widehat{v}_j(u) = u$.
  This means that $\sum_{j=1}^n \lambda_j U_j$
  is relatively weakly ($=$norm) open in $B_X$.
\end{proof}

\begin{prop}\label{prop:extB_X_and_points_in_B^0_X_have_P11}
  Let $X$ be a Banach space.
  \begin{enumerate}
  \item\label{item:prop_P11-1}
    Suppose that, for a point $x\in B_X$, either $x$ is an extreme
    point of $B_X$ or $\|x\|<1$.
    Then $x$ has property \CCOM.
  \item\label{item:prop_P11-2}
    Suppose that $X$ is strictly convex. Then $X$ has property \CCOM.
  \end{enumerate}
\end{prop}
\begin{proof}
  \ref{item:prop_P11-2} follows immediately from \ref{item:prop_P11-1};
  so let us prove \ref{item:prop_P11-1}.
  Let $n\in\N$, $n\geq2$, let $x_1,\dotsc,x_n \in B_X$ and $\lambda_1,\dotsc,\lambda_n >0$, $\sum_{j=1}^n \lambda_j = 1$,
  be such that $x = \sum_{j=1}^n \lambda_j x_j$, and let $\eps>0$.

  First suppose that $x$ is an extreme point of $B_X$.
  Then $x = x_j \in S_X$ for every $j$.
  Taking $\delta:=\eps$ and
  defining $\widehat{v}_j(u)= u$ for every
  $u\in B:=B(x,\delta)\cap B_X$,
  the conditions
  \eqref{eq:P11-conditions}
  hold.

  Now suppose that $\|x\| = 1 - \sigma$ for some $\sigma > 0$.
  Put $r := \frac\varepsilon2$.
  Choose $\delta > 0$ with $\delta < \sigma r$.
  Define, for
  every $u \in B := B(x,\delta)\cap B_X$,
  \begin{equation*}
    \widehat{v}_j(u) = x_j + r ( x - x_j ) + (u - x),
    \quad j \in \{1,\dotsc,n\}.
  \end{equation*}
  Since $\sum_{j=1}^n \lambda_j (x - x_j) = 0$,
  we get $\sum_{j=1}^n \lambda_j\widehat{v}_j(u) = u$.
  We also have
  \begin{equation*}
    \|\widehat{v}_j(u)\| = \|r x + (1-r)x_j + (u-x)\|
    \le r(1-\sigma) + (1-r) + \delta
    = 1 - r \sigma + \delta \le 1;
  \end{equation*}
  hence $\widehat{v}_j : B \to B_X$.
  Since $\delta < \sigma r$, we have
  \begin{equation*}
    \|\widehat{v}_j(u) - x_j\| = \|r(x - x_j) + (u - x)\|
    \le r(1-\sigma) +r +\delta
    < 2r = \varepsilon,
  \end{equation*}
  and we are done.
\end{proof}

Now comes the ``core'' result of this section.

\begin{thm}\label{thm: "core" theorem for C_0(K,X)}
  Let $K$ be a scattered locally compact Hausdorff space
  and let $X$ be a finite dimensional Banach space.
  Suppose that an element $x\in B_{C_0(K,X)}$ is such that, for every
  $t\in K$, the point $x(t)\in B_X$ has property \CCOM.
  Then, whenever $x$ belongs to a finite convex combination of
  relatively weakly open subsets of $B_{C_0(K,X)}$,
  the element $x$ is an interior point of this convex combination in
  the relative weak topology of $B_{C_0(K,X)}$.
\end{thm}

Before proving Theorem~\ref{thm: "core" theorem for C_0(K,X)},
let us cash in some dividends it brings
summarized in the following main theorem.

\begin{thm}\label{thm: main results on CWO-s for C_0(K,X)}
  Let $K$ be a scattered locally compact Hausdorff space
  and let $X$ be a finite dimensional Banach space.
  Then
  \begin{enumerate}
  \item\label{item:1-CWO-C_0KX}
    $C_0(K,X)$ has property \CWOBX;
  \item\label{item:2-CWO-C_0KX}
    if $X$ has property \CCOM,
    then $C_0(K,X)$ has property \CWO;
  \end{enumerate}
\end{thm}
\begin{proof}
  \ref{item:1-CWO-C_0KX} follows from
  Theorem~\ref{thm: "core" theorem for C_0(K,X)} and
  Proposition~\ref{prop:extB_X_and_points_in_B^0_X_have_P11},
  \ref{item:prop_P11-1}.

  \medskip
  \ref{item:2-CWO-C_0KX} follows trivially from
  Theorem~\ref{thm: "core" theorem for C_0(K,X)}.
\end{proof}

In the proof of Theorem \ref{thm: "core" theorem for C_0(K,X)}, it is convenient to rely on the following lemma.

\begin{lem}\label{lem: C_0(K,X)* when K scattered -- betterment}
  Let $K$ be a scattered locally compact Hausdorff space
  and let $X$ be a Banach space.
  \begin{enumerate}
  \item\label{item:1-helperC0KX}
    Let $f\in S_{C_0(K,X)^\ast}$ and let $\eps>0$.
    Then there are $N\in\N$, $t_1,\dotsc,t_N\in K$, and $\xs_1,\dotsc,\xs_N\in\Xs$
    such that $\sum_{j=1}^N\|\xs_j\|\leq1$ and,
    for the functional
    $g = \sum_{j=1}^N \delta_{t_j}\otimes \xs_j \in C_0(K,X)^\ast$, where
    \begin{equation*}
      g(z) = \sum_{j=1}^N \xs_j \bigl( z(t_j) \bigr),
      \quad z\in C_0(K,X),
    \end{equation*}
    one has $\|f-g\|<\eps$.
  \item\label{item:2-helperC0KX}
    Let $x\in B_{C_0(K,X)}$ and let $U$ be a neighbourhood of $x$
    in the relative weak topology of $B_{C_0(K,X)}$.
    Then there are a finite subset~$T$ of~$K$ and an
    $\eps > 0$ such that, whenever $u \in B_{C_0(K,X)}$ satisfies
    \begin{equation}\label{eq: ||u(t)-x(t)||<delta for every t in T}
      \|u(t) - x(t)\| < \eps
      \quad \text{for every } t\in T,
    \end{equation}
    one has $u\in U$.
  \item\label{item:3-helperC0KX}
    Suppose that $X$ is finite dimensional.
    Let $x\in B_{C_0(K,X)}$, let $T$ be a finite subset of $K$,
    and let $\varepsilon > \nobreak0$.
    Then there is a neighbourhood $U$ of $x$
    in the relative weak topology of $B_{C_0(K,X)}$
    such that every $u\in U$ satisfies
    \eqref{eq: ||u(t)-x(t)||<delta for every t in T}.
  \end{enumerate}
\end{lem}
\begin{proof}
  \ref{item:1-helperC0KX}.
  Set $Z := C_0(K,X) = C_0(K) \iten X$.
  It is known that $Z^* = C_0(K)^* \pten X^*$ and that
  $B_{Z^*} = \clco\{S_{C_0(K)^*} \otimes S_{X^*}\}$
  \cite[Proposition~2.2]{MR2003f:46030}.
  Since $K$ is scattered, we have $C_0(K)^*=\ell_1(K)$ and
  \begin{equation*}
    B_{Z^*} = \clco\{\delta_s \otimes \xs :
    s \in K, \xs \in S_{X^*}\}.
  \end{equation*}

  \medskip
  \ref{item:2-helperC0KX}.
  Let a finite subset $\mathcal{F}\subset S_{C_0(K,X)^\ast}$
  and an $\eps>0$ be such that
  \begin{equation*}
    \{
    u \in  B_{C_0(K,X)} \colon\, |f(u) - f(x)| < 3\eps
    \text{ for every }
    f \in \mathcal{F}
    \}
    \subset U.
  \end{equation*}
  By \ref{item:1-helperC0KX}, for every $f \in \mathcal{F}$,
  there are  $N_f\in\N$, $t_{f,j}\in K$, $\xs_{f,j}\in\Xs$,
  $j \in \{1,\dotsc,N_f\}$,
  such that $\sum_{j=1}^{N_f} \|\xs_{f,j}\| \leq 1$ and,
  for the functional
  $g_f = \sum_{j=1}^{N_f} \delta_{t_{f,j}}\otimes\xs_{f,j} \in C_0(K,X)^\ast$,
  one has $\|f-g_f\|<\eps$.

  Set
  \begin{equation*}
    T: = \bigl\{t_{f,j} \colon\, f\in\mathcal{F},\,
    j \in \{1,\dotsc,N_f\}\,\bigr\},
  \end{equation*}
  and suppose that $u\in B_{C_0(K,X)}$ satisfies
  \eqref{eq: ||u(t)-x(t)||<delta for every t in T}.
  For every $f\in\mathcal{F}$, since
  \begin{align*}
    |g_f(u-x)|
    &=\biggl|
      \sum_{j=1}^{N_f}\xs_{f,j}\bigl(u(t_{f,j})-x(t_{f,j})\bigr)
      \biggr|
      \leq \sum_{j=1}^{N_f} \|\xs_{f,j}\| \|u(t_{f,j})-x(t_{f,j})\|\\
    &< \eps \sum_{j=1}^{N_f}\|\xs_{f,j}\| \leq \eps,
  \end{align*}
  one has
  \begin{align*}
    |f(u)-f(x)|
    &\leq\bigl| f(u)-g_f(u) \bigr| + |g_f(u-x)| + \bigl| g_f(x)-f(x) \bigr|\\
    &\leq 2 \|f-g_f\| +|g_f(u-x)|\\
    &< 2\eps + \eps = 3\eps,
  \end{align*}
  and it follows that $u\in U$.

  \medskip
  \ref{item:3-helperC0KX}.
  Let $B\subset B_{\Xs}$ be a finite
  $\frac{\varepsilon}{3}$-net for $B_{\Xs}$. Set
  \begin{equation*}
    U :=\Bigl\{
    u\in
    B_{C_0(K,X)}\colon\, |(\delta_t\otimes
  \xs)(u-x)| < \dfrac{\varepsilon}{3}
    \quad\text{for all $t\in T$ and $\xs\in B$}
    \Bigr\},
  \end{equation*}
  where the functional $\delta_t\otimes\xs\in C_0(K,X)^\ast$ is defined by
  \begin{equation*}
    (\delta_t\otimes\xs)(u) = \xs\bigl(u(t)\bigr),
    \quad u\in C_0(K,X).
  \end{equation*}
  Let $u \in U$ and $t\in T$ be arbitrary.
  Picking $\zs \in B_{\Xs}$ so that
  \begin{equation*}
    \zs \bigl(u(t) - x(t) \bigr) = \|u(t) - x(t)\|,
  \end{equation*}
  there is an $\xs\in B$ satisfying $\|\zs-\xs\| < \frac{\varepsilon}{3}$.
  One has
  \begin{align*}
    \|u(t)-x(t)\|
    &=\zs\bigl(u(t)-x(t)\bigr)\\
    &\leq\|\zs-\xs\|\,\|u(t)-x(t)\|+\bigl|\xs\bigl(u(t)-x(t)\bigr)\bigr|\\
    &<2\dfrac{\varepsilon}{3}+\dfrac{\varepsilon}{3}=\varepsilon.
  \end{align*}
\end{proof}

We are now in a position to prove
Theorem~\ref{thm: "core" theorem for C_0(K,X)}.

\begin{proof}[Proof of Theorem \ref{thm: "core" theorem for C_0(K,X)}]
  Let $n\in\N$, let $U_1,\dotsc,U_n$
  be relatively weakly open subsets of $B_{C_0(K,X)}$,
  and let $x_j \in U_j$ and $\lambda_j > 0$,
  $\sum_{j=1}^n \lambda_j = 1$,
  be such that $x = \sum_{j=1}^n \lambda_j x_j$.
  We are going to find a neighbourhood $U$ of $x$
  in the relative weak topology of $B_{C_0(K,X)}$
  such that $U \subset \sum_{j=1}^n \lambda_j U_j$.

  By Lemma \ref{lem: C_0(K,X)* when K scattered -- betterment}, (b),
  there are an $\eps>0$ and a finite subset $T$ of $K$ such that
  \begin{itemize}
  \item
    \emph{whenever  $u_1,\dotsc,u_n \in B_{C_0(K,X)}$
      are such that, for every $s\in T$,
      \begin{equation}\label{eq: ||y(s)-x_i(s)||<eps and ||w(s)-z(s)||<eps}
        \|u_j(s) - x_j (s) \| < \eps,\quad j \in \{1,\dotsc,n\},
      \end{equation}
      one has $u_j \in U_j$, $j \in \{1,\dotsc,n\}$.}
  \end{itemize}
  For every $s\in T$, let $\delta_s$, $B_s$, and $\widehat{v}_{s,j}$
  be, respectively, the $\delta$, $B$, and the functions
  $\widehat{v}_j$ from Definition~\ref{def:P11}
  with $x=x(s)$ and $x_j = x_j(s)$.
  By Lemma~\ref{lem: C_0(K,X)* when K scattered -- betterment},~(c),
  there is a neighbourhood $U$ of $x$ in
  the relative weak topology of $B_{C_0(K,X)}$ such that,
  for every $u\in U$,
  \begin{equation*}
    \|u(s) - x(s)\| < \delta_s
    \quad\text{for every } s\in T.
  \end{equation*}
  Let $u \in U$ be arbitrary.
  We are going to show that $u \in \sum_{j=1}^n \lambda_j U_j$.

  For every $s\in T$,
  pick $H_s$ to be a compact neighbourhood of $s$ such that
  \begin{equation*}
    u(t) \in B_s \quad \text{for every } t\in H_s.
  \end{equation*}
  We can choose the neighbourhoods $H_s$, $s\in T$, to be pairwise
  disjoint.
  For every $j \in \{1,\dotsc,n\}$, since $X$ is finite dimensional,
  by Tietze's extension theorem,
  there is a continuous function $w_j : K \to X$ such that
  $w_j(t) = \widehat{v}_{s,j} \bigl( u(t) \bigr)$ for every $s\in T$
  and every $t\in H_s$.
  By Urysohn's lemma,
  there is a $\kappa\in C_0(K,\R)$ with values in $[0,1]$
  such that $\kappa|_T=1$ and
  $\supp\kappa \subset \bigcup_{s\in T} H_s$.
  Set
  \begin{equation*}
    u_j := \kappa w_j + (1-\kappa) u\in B_{C_0(K,X)},\quad j \in \{1,\dotsc,n\}.
  \end{equation*}
  Notice that $u = \sum_{j=1}^n \lambda_j u_j$.
  Indeed, if $t \notin \supp\kappa$,
  then $\kappa(t) = 0$ and thus
  $\sum_{j=1}^n \lambda_j u_j(t) =
  \sum_{j=1}^n \lambda_j u(t) = u(t)$;
  if $t \in \supp\kappa$,
  then $t \in H_s$ for some $s \in T$,
  thus
  \begin{equation*}
    \sum_{j=1}^n \lambda_j w_j(t)
    =
        \sum_{j=1}^n \lambda_j \widehat{v}_{s,j} \bigl( u(t) \bigr)
    =  u(t)
  \end{equation*}
  (because $u(t)\in B_s$),
  and
  \begin{align*}
    \sum_{j=1}^n \lambda_j u_j(t)
    &=
    \sum_{j=1}^n \lambda_j \bigl(\kappa(t) w_j(t) + (1-\kappa(t))u(t) \bigr) \\
    &=\kappa(t)u(t)+\bigl(1-\kappa(t)\bigr)u(t)
      =u(t).
  \end{align*}
  Also notice that, for every $j \in \{1,\dotsc,n\}$ and every $s\in T$, since
  $u_j(s) = w_j(s) = \widehat{v}_{s,j} \bigl( u(s) \bigr)$
  and $u(s)\in B_s$, one has
  \eqref{eq: ||y(s)-x_i(s)||<eps and ||w(s)-z(s)||<eps},
  thus $u_j \in U_j$.
\end{proof}

\section{Banach spaces with property \CCOM}
\label{sec:4}

In this section we explore Banach spaces with property \CCOM.
We give an example of a finite dimensional Banach space,
which fails property \CCOM, and many examples of
finite dimensional Banach spaces with property \CCOM
(see Propositions~\ref{prop:polyhedral-dual-ccom}
and \ref{prop:every-2-dim-is-CCOM}, and
Theorem~\ref{thm:complex-ell_1^n} below).
We start with a characterization of property \CCOM.

\begin{defn}\label{defn:P10}
  Let $X$ be a Banach space. We say that a point $x\in B_X$ has
  \emph{property \CCOMtwo} if the condition \CCOMn
  in Definition~\ref{def:P11} is satisfied for $n=2$.

  We say that the space $X$ has \emph{property \CCOMtwo}
  if every point $x\in B_X$ has property \CCOMtwo.
\end{defn}

\begin{prop}\label{prop: X has (co) <=> X has (co2)}
  For a Banach space $X$,
  properties \CCOMtwo and \CCOM are equivalent.
\end{prop}
\begin{proof}
  It is clear that property \CCOM for $X$ implies property \CCOMtwo.
  For the reverse implication, assume that $X$ has property \CCOMtwo,
  and that $m \in \N$ with $m \geq 2$ is such that,
  whenever $x \in B_X$, the condition \CCOMn
  in Definition~\ref{def:P11} holds for $n=m$.
  It suffices to show that, whenever $x \in B_X$,
  the condition~\CCOMn holds also for $n = m+1$.
  To this end, let $x\in B_X$, let  $x_1,\dotsc,x_{m+1}\in B_X$ and
  $\lambda_1,\dotsc,\lambda_{m+1}>0$, $\sum_{j=1}^{m+1} \lambda_j = 1$,
  be such that $x = \sum_{j=1}^{m+1} \lambda_j x_j$, and let $\eps>0$.

  Setting $\lambda := \sum_{j=2}^{m+1} \lambda_j$ and
  $y := \sum_{j=2}^{m+1} \frac{\lambda_j}{\lambda} x_j$,
  observe that $\sum_{j=2}^{m+1} \frac{\lambda_j}{\lambda} = 1$,
  $\lambda_1 + \lambda = 1$, and $x = \lambda_1 x_1 + \lambda y$.

  By our assumption, there is a $\delta_0 > 0$ such that,
  setting $B_0 := B(y,\delta_0) \cap B_X$,
  there are continuous functions
  $\widehat{v}_j \colon \, B_0 \to B_X$,
  $j \in \{2,\dotsc, m+1\}$, such that, for every $v \in B_0$,
  \begin{equation*}
    v = \sum_{j=2}^{m+1}\frac{\lambda_j}{\lambda} \widehat{v}_j(v)
    \quad\text{and}\quad
    \|\widehat{v}_j(v) - x_j\| < \eps
    \quad\text{for every } j \in \{2,\dotsc,m+1\}.
  \end{equation*}

  Since $X$ has property \CCOMtwo,
  there is a $\delta > 0$ such that,
  setting $B := B(x,\delta) \cap B_X$,
  there are continuous functions $\widehat{u}_1 \colon \,B \to B_X$
  and $\widehat{v} \colon\, B \to B_X$ such that,
  for every $u \in B$,
  \begin{equation*}
    u = \lambda_1\widehat{u}_1(u) + \lambda \widehat{v}(u),
    \quad
    \|\widehat{u}_1(u) - x_1\| < \eps,
    \quad \|\widehat{v}(u) - y\| < \delta_0.
  \end{equation*}
  It remains to define, for every $j \in \{2,\dotsc,m+1\}$,
  a function $\widehat{u}_j \colon\, B \to B_X$
  by $\widehat{u}_j = \widehat{v}_j \circ \widehat{v}$,
  because in that case, for every $u \in B$,
  observing that $\widehat{v}(u) \in B_0$, one has
  \begin{equation*}
    \|\widehat{u}_j(u) - x_j\| =
    \bigl\|\widehat{v}_j\bigl(\widehat{v}(u)\bigr) - x_j\bigr\|
    < \eps
    \quad\text{for every } j \in\{2,\dotsc,m+1\},
  \end{equation*}
  and
  \begin{equation*}
    \widehat{v}(u) =
    \sum_{j=2}^{m+1} \frac{\lambda_j}{\lambda}\,
    \widehat{v}_j\bigl(\widehat{v}(u)\bigr)
    = \sum_{j=2}^{m+1} \frac{\lambda_j}{\lambda}\,
    \widehat{u}_j(u),
  \end{equation*}
  and thus
  \begin{equation*}
    u = \lambda_1 \widehat{u}_1(u) + \lambda \widehat{v}(u)
    = \sum_{j=1}^{m+1} \lambda_j \widehat{u}_j(u),
  \end{equation*}
  which shows that $X$ has property \CCOM.
\end{proof}

The following proposition indicates a class of Banach spaces $X$, which
do not have property \CCOM nor does the space $C_0(K,X)$ have property \CWOSX.
A concrete example of a representative of this class will be given
in Example~\ref{example: a space X without (P10) such that C_0(K,X) is not CWO1}.

\begin{prop}\label{prop: a class of Banach spaces for which C_0(K,X) is never CWO1 (and X fails (P10))}
  Let $X$ be a Banach space.
  Suppose that there is a point $x \in S_X$ such that
  $x \notin \ext B_X$ and
  $\ext B_X \cap B(x,\varepsilon) \neq \emptyset$
  for every $\varepsilon > 0$.
  Then
  \begin{enumerate}
  \item\label{item:1-linesegment-in-sphere}
    $x$ fails property \CCOMtwo;
  \item\label{item:2-linesegment-in-sphere}
    $X$ fails property \CWOSX;
  \item\label{item:3-linesegment-in-sphere}
    whenever $K$ is a locally compact Hausdorff space, the
    space $C_0(K,X)$ fails property \CWOSX.
  \end{enumerate}
\end{prop}
\begin{proof}
  \ref{item:1-linesegment-in-sphere}
  is obvious.

  \medskip
  \ref{item:2-linesegment-in-sphere}.
  By assumption there are $x_1, x_2 \in S_X$,
  $x_1 \neq x_2$, such that
  $x = \frac{1}{2}x_1 + \frac{1}{2}x_2$.
  Define a linear functional
  $g : \spann\{x, x_1 - x_2\} \to \K$ by
  $g(x)=0$ and $g(x_1 - x_2)=1$
  (observe that the elements $x$
  and $x_1 - x_2$ are linearly independent).
  Letting $x^* \in X^*$ be any norm preserving extension of
  $\frac{g}{\|g\|}$, one has $x^* \in S_{X^*}$, $x^*(x) = 0$,
  and $x^*(x_1) = 2\alpha$, $x^*(x_2) = -2\alpha$ for some $\alpha > 0$.
  Consider the slices
  \begin{equation*}
    S_1 := \{a \in B_X : \re x^*(a) > \alpha\}
    \quad \text{and} \quad
    S_2 := \{a \in B_X : \re(-x^*)(a) > \alpha\}.
  \end{equation*}
  Then $x_j \in S_j$, $j \in \{1,2\}$, and thus
  $x \in \frac{1}{2}S_1 + \frac{1}{2}S_2$.

  Let $U$ be an arbitrary neighbourhood of $x$ in the relative weak
  topology of $B_X$.
  Then there is a $\delta > 0$ such that
  $B(x, \delta) \cap B_X \subset U$.
  We may assume that $\delta < \alpha$.

  By assumption,
  there exists $u \in \ext B_X \cap B(x,\delta) \subset U$.
  Suppose that $u = \frac{1}{2}u_1 + \frac{1}{2}u_2$
  with $u_j \in S_j$, $j \in \{1,2\}$. Since $u$ is an extreme point,
  $u_1 = u_2 = u$. But
  \begin{equation*}
    |\re x^*(u)| \le |\re x^*(x)| + \|u - x\|
    < \delta < \alpha,
  \end{equation*}
  hence $u_j \notin S_j$, $j \in \{1,2\}$, and
  $u \notin \frac{1}{2} S_1 + \frac{1}{2} S_2$.
  Thus $x$ is not an interior point of
  $\frac{1}{2} S_1 + \frac{1}{2} S_2$
  in the relative weak topology of $B_X$.

  \medskip
  \ref{item:3-linesegment-in-sphere}
  follows from
  \ref{item:2-linesegment-in-sphere}
  and Proposition~\ref{prop:norm1compl-inherit-cwox}
  since $C_0(K,X) = C_0(K) \iten X$ contains
  a one complemented copy of $X$.
\end{proof}

We now give a concrete example
of the phenomenon described in Proposition \ref{prop: a class of Banach spaces for which C_0(K,X) is never CWO1 (and X fails (P10))}.

\begin{example}\label{example: a space X without (P10) such that C_0(K,X) is not CWO1}
  Let $X$ be the Banach space $\R^3$ whose unit ball is
  \begin{align*}
    B_X =\conv \bigl((B_{\ell_2^3} - e_1) \cup B_{\ell_\infty^3} \cup (B_{\ell_2^3} + e_1)\bigr),
  \end{align*}
  where $e_1 = (1, 0, 0)$. Then the point $(1,0,1)\in S_X$
  is not an extreme point of $B_X$ (because it lies on the line segment
  connecting the points $(1,-1,1)\in S_X$ and $(1,1,1)\in S_X$),
  but it has extreme points of $B_X$ arbitrarily  close to it.
\end{example}

Let $X$ be a finite dimensional Banach space.
If $X$ is a real Banach space, then
$X$ is called \emph{polyhedral} if $\ext B_X$ is a finite set.
For a complex Banach space $X$,
following \cite{MR0454597}, we say that
$B_X$ is a \emph{complex polytope} if
there exists a finite set $A \subset \ext B_X$
such that $\ext B_X = \T \cdot A$.
We will say that
$X$ is \emph{polyhedral} if $B_X$ is a complex polytope.

\begin{prop}\label{prop:polyhedral-dual-ccom}
  Let $X$ be a finite dimensional Banach space
  such that the dual $X^*$ is polyhedral.
  Then $X$ has property \CCOMtwo (and hence property~\CCOM).
\end{prop}

\begin{proof}
  Let $\varepsilon > 0$ and define $r := \frac{\varepsilon}{2}$.
  Let $\ext B_{X^*} / \T = \{\phi_1,\ldots,\phi_m\}$.

  Let $x \in B_X$.
  Assume that $x_1, x_2 \in B_X$ and $\lambda_1, \lambda_2 > 0$,
  with $\lambda_1 + \lambda_2 = 1$, are such that
  $x = \lambda_1 x_1 + \lambda_2 x_2$.
  Define $J := \{n : |\phi_n(x)| < 1\}$ and
  $\sigma := \min_{n \in J} (1 - |\phi_n(x)|) > 0$.
  Choose $\delta > 0$ such that
  $\delta < r\sigma$.
  Set $B := B(x,\delta) \cap B_X$ and
  define functions $\widehat{v}_1,\widehat{v}_2 : B \to B_X$  by
  \begin{equation*}
    \widehat{v}_j(u) := x_j + r ( x - x_j ) + (u - x),\quad j \in \{1,2\}.
  \end{equation*}
    It is trivial that $\widehat{v}_1$ and $\widehat{v}_2$
  are continuous.
  We have
  \begin{equation*}
    \lambda_1 \widehat{v}_1(u)
    +
    \lambda_2 \widehat{v}_2(u)
    =
    x + r ( x - x ) + (u - x) = u
  \end{equation*}
  and
  \begin{equation*}
    \|\widehat{v}_j(u) - x_j\| = \|r(x - x_j) + (u - x)\|
    \le r(1-\sigma) +r +\delta
    < 2r = \varepsilon.
  \end{equation*}

  It remains to show that $\widehat{v}_1,\widehat{v}_2(u) \in B_X$
  for all $u\in B$. Let $j \in \{1,2\}$, $u\in B$, $n\in\{1,\dotsc,m\}$,
  and $\alpha\in\T$ be arbitrary.

  If $|\phi_n(x)| = 1$, then $|(\alpha \phi_n)(x)| = 1$
  and
  \begin{equation*}
    (\alpha \phi_n)(x)
    = \lambda_1 (\alpha \phi_n)(x_1)
    + \lambda_2 (\alpha \phi_n)(x_2).
  \end{equation*}
  This means that $(\alpha \phi_n)(x) \in \T$
  has been written as a convex combination of elements
  in $\D$.
  But every point in $\T$ is an extreme point
  in $\D$,
  hence $(\alpha \phi_n)(x_j) = (\alpha \phi_n)(x)$, $j \in \{1,2\}$.
  Since
  \begin{equation*}
    \widehat{v}_j(u)
    = (1-r)(x_j - x) + u,
  \end{equation*}
  we have
  \begin{equation*}
    |(\alpha \phi_n)(\widehat{v}_j(u))|
    \le (1-r)|(\alpha \phi_n)(x_j - x)|
    + |\phi_n(u)|
    = 0 + |\phi_n(u)| \le 1.
  \end{equation*}

  If $|\phi_n(x)| < 1$, then we use the fact that
  \begin{align*}
    \widehat{v}_j(u)
    = (1-r)x_j + r x + (u - x)
  \end{align*}
  and get
  \begin{align*}
    |(\alpha \phi_n)(\widehat{v}_j(u))|
    &\le (1-r)|\phi_n(x_j)| + r|\phi_n(x)| + |\phi_n(u - x)| \\
    &\le (1-r) + r(1-\sigma) + \delta \\
    &\le 1 - r\sigma + r\sigma = 1.
  \end{align*}

  In conclusion, $|\phi(\widehat{v}_j(u))| \le 1$
  for all $\phi \in \ext B_{X^*}$;
  hence $\widehat{v}_j(u) \in B_X$.
\end{proof}

Note that both real and complex $\ell_1^n$ are polyhedral,
and while real $\ell_\infty^n$ is polyhedral,
complex $\ell_\infty^n$ is not.
We will however prove that complex $\ell_1^n$
has property \CCOM in Theorem~\ref{thm:complex-ell_1^n}
below.

Recall that, by
Proposition~\ref{prop:extB_X_and_points_in_B^0_X_have_P11},
\ref{item:prop_P11-1},
every norm-less-than-one point of any Banach space has property
\CCOM (and hence property \CCOMtwo).
Next we give a necessary and
sufficient condition for norm-one points in Banach spaces
to have property \CCOMtwo, which is easier to verify that the condition
from Definition~\ref{defn:P10} (and Definition~\ref{def:P11}).
More precisely, we show that it is enough to
define the functions from Definition~\ref{def:P11}
on a neighbourhood of the norm-one point on the sphere
and not a neighbourhood in the unit ball.
This result will be applied to show that,
for any $n \in \N$,
the complex space $\ell_1^n$ has property \CCOMtwo.

\begin{thm}\label{thm:P10-for-S_X}
  Let $X$ be a Banach space and let $x\in S_X$. The following
  assertions are equivalent:
  \begin{enumequi}
  \item\label{item:1-thm-P10-S_X}
    $x$ has property \CCOMtwo;
  \item\label{item:2-thm-P10-S_X}
    whenever $x_1, x_2 \in S_X$, $x_1 \ne x_2$, and
    $\lambda_1, \lambda_2 > 0$, $\lambda_1 + \lambda_2 = 1$,
    are such that $x = \lambda_1 x_1 + \lambda_2 x_2$,
    and $\eps>0$,
    there is a $\delta>0$ such that, setting $S = B(x,\delta)\cap
    S_X$, there are continuous functions
    $\mathfrak{v}_1,\mathfrak{v}_2 : S \to B_X$
    such that, for every $u \in S$,
    \begin{equation}\label{eq:P10-S_X-conditions}
      u = \lambda_1 \mathfrak{v}_1(u)
      + \lambda_2 \mathfrak{v}_2(u)
      \quad \text{and} \quad
      \|\mathfrak{v}_j(u) - x_j\| < \eps,
      \quad j \in \{1,2\}.
    \end{equation}
  \end{enumequi}
\end{thm}
\begin{proof}
  \ref{item:1-thm-P10-S_X} $\Rightarrow$ \ref{item:2-thm-P10-S_X}
  is obvious.

  \medskip
  \ref{item:2-thm-P10-S_X} $\Rightarrow$ \ref{item:1-thm-P10-S_X}.
  Let $x_1, x_2 \in B_X$ and $\lambda_1, \lambda_2 > 0$,
  $\lambda_1 + \lambda_2 = 1$, be such that
  $x = \lambda_1 x_1 + \lambda_2 x_2$,
  and let $0 < \eps < 1$.
  Then, in fact, $x_1, x_2 \in S_X$.

  First consider the case when $x_1 = x_2$; then also $x = x_1$.
  Taking $\delta = \eps$ and defining
  $\widehat{v}_1(u) =\widehat{v}_2(u)= u$
  for every $u \in B := B(x,\delta) \cap B_X$,
  the conditions \eqref{eq:P11-conditions}
  hold, hence $x$ has property \CCOMtwo.

  Now suppose that $x_1 \ne x_2$; then, in fact, $x_1 \ne x \ne x_2$.
  By our assumption, there is a $\gamma \in (0,\eps)$ such that,
  setting $S := B(x,\gamma) \cap S_X$,
  there are continuous functions
  $\mathfrak{v}_1,\mathfrak{v}_2 : S \to B_X$
  satisfying
  \eqref{eq:P10-S_X-conditions}  with $\eps$ replaced by~$\frac{\eps}{2}$
  for every $u \in S$.

  Set $C := \{\alpha u \colon\,\alpha \in [0,1], \,u \in S\}$
  and $\delta = \frac{\gamma}{4}$.
  Observe that $B := B(x,\delta) \cap B_X \subset C$.
  Indeed, suppose that $a \in B$. Since
  \begin{equation*}
    \delta > \|x-a\| \geq \|x\| - \|a\| = 1 - \|a\|,
  \end{equation*}
  one has $\|a\| > 1 - \delta > \frac{1}{2}$.
  For $u := \frac{a}{\|a\|}$,
  one has $a = \|a\|u$ and $u \in S$, because
  \begin{align*}
    \|u-x\|
    &\leq
      \Bigl\| \dfrac{a}{\|a\|} - \dfrac{x}{\|a\|} \Bigr\|
      + \Bigl\| \dfrac{x}{\|a\|} - x \Bigr\|
      = \dfrac{\|a-x\|}{\|a\|} + \dfrac{1-\|a\|}{\|a\|}
      < \dfrac{2\delta}{\|a\|}\\
    &< 4 \delta
      = \gamma.
  \end{align*}

  Since every $a \in B$ has a unique representation $a = \alpha u$,
  where $\alpha \in (0,1]$ and $u \in S$,
  the functions $\widehat{v}_1,\widehat{v}_2 : B \to B_X$ defined by
  \begin{equation}\label{eq:homogeneous_def_of_widehatv_j}
    \widehat{v}_j(a) =
    \widehat{v}_j(\alpha u)
    := \alpha \mathfrak{v}_j(u),
    \quad
    j \in \{1,2\},
  \end{equation}
  are well defined.
  We now show that these functions are continuous.
  To this end, let $\alpha_0 u_0 \in B$
  ($\alpha_0 \in (0,1]$, $u_0 \in S$)
  and $\beta > 0$.
  By the continuity of $\mathfrak{v}_1$ and $\mathfrak{v}_2$,
  there is a $\delta_0 > 0$ such that,
  whenever $u \in S$ satisfies $\|u - u_0\| < \delta_0$,
  one has
  $\|\mathfrak{v}_j(u) - \mathfrak{v}_j(u_0)\| < \frac{\beta}{2}$,
  $j \in \{1,2\}$.
  Suppose that $\alpha u \in B$
  ($\alpha \in (0,1]$, $u \in S$) is such that
  $\|\alpha u - \alpha_0 u_0\| <
  \min \bigl\{ \frac{\delta_0}{4}, \frac{\beta}{2} \bigr\}$.
  Then also
  \begin{equation*}
    \min \Bigl\{ \dfrac{\delta_0}{4}, \dfrac{\beta}{2}\Bigr\}
    > \|\alpha u - \alpha_0 u_0\|
    \geq \bigl| \|\alpha u\| - \|\alpha_0 u_0\| \bigr|
    =| \alpha - \alpha_0 |
  \end{equation*}
  and, since $\alpha = \|\alpha u\| > \frac{1}{2}$,
  \begin{align*}
    \|u-u_0\|
    &\leq \frac{1}{\alpha}
      \bigl(\| \alpha u - \alpha_0 u_0\|
      +|\alpha_0 - \alpha|\,\|u_0\|\bigr)
      < 2 \Bigl(\dfrac{\delta_0}{4} + \dfrac{\delta_0}{4}\Bigr)
      = \delta_0.
  \end{align*}
  Thus
  \begin{align*}
    \|\widehat{v}_j(\alpha u) - \widehat{v}_j(\alpha_0 u_0)\|
    &=
    \|\alpha \mathfrak{v}_j(u) - \alpha_0 \mathfrak{v}_j(u_0)\|
    \\
    &\leq
      |\alpha - \alpha_0|\,\|\mathfrak{v}_j(u)\|
      + \alpha_0\| \mathfrak{v}_j(u) - \mathfrak{v}_j(u_0)\|
      < \frac{\beta}{2} + \frac{\beta}{2}
      = \beta.
  \end{align*}
  It follows that the functions
  $\widehat{v}_1,\widehat{v}_2 : B \to B_X$ are continuous.

  It remains to observe that,
  whenever $\alpha u \in B$ ($\alpha \in (0,1]$, $u \in S$), one has
  \begin{equation*}
    \lambda_1(\alpha \mathfrak{v}_1(u))
    +
    \lambda_2(\alpha \mathfrak{v}_2(u))
    = \alpha
    (\lambda_1\mathfrak{v}_1(u) + \lambda_2\mathfrak{v}_2(u))
    = \alpha u,
  \end{equation*}
  and, since $\alpha = \|\alpha u\| > 1 - \delta$,
  \begin{align*}
    \|\alpha \mathfrak{v}_j(u) - x_j\|
    \leq (1 - \alpha) \|\mathfrak{v}_j(u)\|
    + \|\mathfrak{v}_j(u) - x_j\|
    < \delta + \frac{\eps}{2}
    < \eps,\quad j \in \{1,2\}.
  \end{align*}
\end{proof}

\begin{prop}\label{prop:every-2-dim-is-CCOM}
Let $X$ be a two dimensional real Banach space. Then $X$ has property \CCOM.
\end{prop}
\begin{proof}
We are going to apply Proposition \ref{prop: X has (co) <=> X has (co2)}
teamed with Theorem~\ref{thm:P10-for-S_X}.
Let $x,x_1,x_2\in S_X$ with $x_1\ne x_2$ and $\lambda_1,\lambda_2>0$
with $\lambda_1+\lambda_2=1$ be such that $x=\lambda_1 x_1+\lambda_2 x_2$,
and let $\eps>0$. We may assume that $d:=\|x-x_1\|\leq\|x-x_2\|$
(or, equivalently, $\lambda_2\leq\lambda_1$)
and that $\eps<d$. Set $a:=\frac{x-x_1}{\|x-x_1\|}$.
Observe that $d\leq1$, and $\|x+ta\|=1$ whenever $|t|\leq d$.
We shall make use of the following claim which is easy to believe
and not much harder to prove.

\begin{Claim}
There is a $\gamma>0$ such that, whenever $0<\delta\leq\gamma$, one has
\begin{equation}\label{eq: B(x,delta) cap S_X = ...}
S_\delta:=B(x,\delta)\cap S_X=\{x+ta\colon t\in(-\delta,\delta)\}.
\end{equation}
\end{Claim}

\noindent%
Letting $0<\delta<\min\{\gamma,\frac{\lambda_1\eps}{2},
\frac{\lambda_2\eps}{2}\}$, where $\gamma>0$ comes from Claim,
we can now define functions $\widehat{v}_1,\widehat{v}_2\colon S_\delta\to B_X$
by
\[
\widehat{v}_1(x+ta)=x_1+\frac{\delta}{\lambda_1}a+ta,
\quad
\widehat{v}_2(x+ta)=x_2-\frac{\delta}{\lambda_2}a+ta,
\qquad t\in(-\delta,\delta).
\]

It remains to prove Claim.
First observe that the elements $x$ and $a$ are linearly independent.
Since all norms on $X$ are equivalent,
there is a $\gamma>0$ such that
\[
D:=\bigl\{b_{st}:=sx+ta\colon
s,t\in\bigl[-\tfrac{d}{2},\tfrac{d}{2}\bigr]\bigr\}
\supset \gamma B_X
\]
(observe that $\gamma\leq\tfrac{d}{2}$).
Now suppose that $0<\delta\leq\gamma$.
Since $B(x,\delta)=x+\delta B_X\subset x+D$,
every $z\in B(x,\delta)$ can be represented as $z=x+b_{st}$,
where $s,t\in[-\frac{d}{2},\frac{d}{2}]$.
Since $\bigl|\frac{t}{1+s}\bigr|\leq d$, one has
\[
\|x+b_{st}\|=\|(1+s)x+ta\|=(1+s)\|x+\tfrac{t}{1+s}\,a\|=1+s,
\]
hence $\|x+b_{st}\|>1$ if $s>0$, and $\|x+b_{st}\|<1$ if $s<0$; thus
\[
B(x,\delta)\cap S_X
\subset\bigl\{x+ta\colon t\in\bigl[-\tfrac{d}{2},\tfrac{d}{2}\bigr]\bigr\}.
\]
Since $\{x+ta\colon |t|\geq\delta\}\cap B(x,\delta)=\emptyset$ and
$\{x+ta\colon |t|<\delta\}\subset B(x,\delta)\cap S_X$,
the equality \eqref{eq: B(x,delta) cap S_X = ...} follows.
\end{proof}

We have already seen, in Proposition~\ref{prop:polyhedral-dual-ccom},
that $\ell_1^n$ over the real scalars has property \CCOM.
Next we show that this is also true for
complex $\ell_1^n$.

\begin{thm}\label{thm:complex-ell_1^n}
  Let $n \in \N$.
  Then the complex space $\ell_1^n$ has property \CCOM.
\end{thm}
\begin{proof}
We are going to apply Proposition \ref{prop: X has (co) <=> X has (co2)}
teamed with Theorem~\ref{thm:P10-for-S_X}.
Let $x_0,x_1,x_2\in S_{\ell_1^n}$ with $x_1\ne x_2$ and $\lambda_1,\lambda_2>0$
with $\lambda_1+\lambda_2=1$ be such that $x_0=\lambda_1 x_1+\lambda_2 x_2$,
and let $\eps>0$. For a complex number $\zeta$, we write $\zeta=(r,\phi)$,
where $r$ and $\phi$ are, respectively, the modulus and an argument of~$\zeta$.
For every $j \in \{0,1,2\}$, let $x_j=(x^j_i)_{i=1}^n$,
where $x^j_i=(r^j_i,\phi^j_i)$.
We may assume that  $\phi^0_i=\phi^1_i=\phi^2_i$
for every $i\in\{1,\dotsc,n\}$
(this is because $\|x_1+x_2\|=\|x_1\|+\|x_2\|$ yields $|x^1_i+x^2_i|=|x^1_i|+|x^2_i|$).

For every $j \in \{0,1,2\}$ and every $\gamma>0$, define
\begin{multline*}
D_j(\gamma)
:=
\bigl\{
\bigl((r_i,\phi_i)\bigr)_{i=1}^n\in \ell^n_1\colon
\text{$|r_i-r^j_i|<\gamma$ for every $i\in\{1,\dotsc,n\}$}\\
\text{and $|\phi_i-\phi^j_i|<\gamma$ for every $i\in\{1,\dotsc,n\}$
 with $r^j_i\ne0$}
\bigr\},
\end{multline*}
and pick a $\gamma>0$ such that $D_j(2\gamma)\subset B(x_j,\eps)$
whenever $j \in \{1,2\}$.
We may assume that $\gamma<\pi$,
and that $2\gamma<r^j_i$ whenever $j \in \{0,1,2\}$ and $i \in \{1,\dotsc,n\}$
are such that $r^j_i>0$.

Set
\begin{align*}
I_1&:=\bigl\{i\in\{1,\dotsc,n\}\colon r^1_i>r^2_i \bigr\},
&I_2&:=\bigl\{i\in\{1,\dotsc,n\}\colon r^1_i<r^2_i \bigr\}.
\end{align*}
Define $\rho^1_i:=\rho^2_i:=r^1_i=r^2_i$ if $i\notin I_1\cup I_2$, and
\begin{align*}
\rho^1_i&:=r^1_i-\frac{\gamma\lambda_2}{|I_1|}
&&\text{and}
&\rho^2_i&:=r^2_i+\frac{\gamma\lambda_1}{|I_1|}
&&\quad\text{for every $i\in I_1$,}\\
\rho^1_i&:=r^1_i+\frac{\gamma\lambda_2}{|I_2|}
&&\text{and}
&\rho^2_i&:=r^2_i-\frac{\gamma\lambda_1}{|I_2|}
&&\quad\text{for every $i\in I_2$.}
\end{align*}
Observe that
\begin{enumprop}
\item
$\lambda_1\rho^1_i+\lambda_2\rho^2_i=\lambda_1 r^1_i+\lambda_2 r^2_i=r^0_i$
for every $i\in\{1,\dotsc,n\}$;
\item
$\sum_{i=1}^n\rho^j_i=\sum_{i=1}^n r^j_i=1$ for every $j \in \{1,2\}$;
\item
$|\rho^j_i-r^j_i|<\gamma$ for every $j \in \{1,2\}$
and every $i \in \{1,\dotsc,n\}$;
\item
$\rho^j_i\geq\frac{\gamma\min\{\lambda_1,\lambda_2\}}{n}=:\beta>0$
for every $j \in \{1,2\}$ and every $i \in \{1,\dotsc,n\}$ with $r^0_i\ne0$.
\end{enumprop}

Choose a $\delta>0$ such that $B(x_0,\delta)\subset D_0(\beta)$.
For every $u\in S:=B(x_0,\delta)\cap S_{\ell^n_1}$,
writing $u:=\Bigl(\bigl(r^0_i+\delta_i(u),\phi_i(u)\bigr)\Bigr)_{i=1}^n$,
where $|\delta_i(u)|<\beta$ and $|\phi_i(u)-\phi^0_i|<\beta$
(here the latter inequality is dropped if $r^0_i=0$), define
\[
\mathfrak{v}_j(u)
:=\Bigl(\bigl(\rho^j_i+\delta_i(u),\phi_i(u)\bigr)\Bigr)_{i=1}^n,
\quad j \in \{1,2\}.
\]
Note that
$\|u\| = \sum_{i=1}^n r_i^0 + \delta_i(u)
= 1 + \sum_{i=1}^n \delta_i(u)$,
hence $\sum_{i=1}^n \delta_i(u) = 0$
and thus
\begin{equation*}
  \|\mathfrak{v}_j(u)\|
  = \sum_{i=1}^n |\rho_i^j + \delta_i(u)|
  = \sum_{i=1}^n \rho_i^j + \delta_i(u) = 1.
\end{equation*}
The functions $\mathfrak{v}_1,\mathfrak{v}_2\colon S \to B_X$ are continuous
and satisfy \eqref{eq:P10-S_X-conditions} for every $u \in S$.
\end{proof}

From Theorem~\ref{thm:complex-ell_1^n}
(Proposition~\ref{prop:polyhedral-dual-ccom} in the real case),
Proposition~\ref{prop:fin_dim_co_cwo}, and
Theorem~\ref{thm: main results on CWO-s for C_0(K,X)}
we know that, for any scattered compact $K$,
$C_0(K,\ell_1^n) = C_0(K) \iten \ell_1^n$
has property \CWO. A similar result does not hold
for the projective tensor product.

\begin{prop}
  Let $X$ be a Banach space.
  Then $X\pten \ell_1^n$ fails property \CWOBX.
\end{prop}

\begin{proof}
  The proof of \cite[Proposition~2.1]{2017arXiv170302919H}
  shows that if $X$ and $Y$ are Banach spaces, and
  $Z := X \oplus_p Y$, where $1 \le p < \infty$,
  then there exists a finite convex combination of slices of $B_Z$
  which contains $0$, but which fails to contain a relatively
  weakly open neighbourhood of $0$.

  The assertion follows since $X\pten \ell_1^n$ is isometrically isomorphic to
  $\ell_1^n(X)$ (see proof of \cite[Example~2.6, p.~19]{MR2003f:46030}).
\end{proof}

\section{Stability results}
\label{sec:2}
In this section we discuss stability results of the \CWO-properties in
Definition \ref{defn:cwos}. We start by showing that they all are
stable by taking one complemented subspaces, but first, let us make things
easier for ourselves.

\begin{lem}\label{lem: reducing to combinations of two}
  Let $X$ be a Banach space.
  \begin{enumerate}
  \item\label{item:a-reduce-lemma}
    The following assertions are equivalent:
    \begin{enumequi}
    \item\label{item:a-1-reduce}
      $X$ has property \CWO;
    \item\label{item:a-2-reduce}
      whenever $U_1$ and $U_2$ are
      relatively weakly open subsets of $B_X$
      and $\lambda_1, \lambda_2 > 0$,
      $\lambda_1 + \lambda_2 = 1$,
      the convex combination $\lambda_1 U_1 + \lambda_2 U_2$
      is open in the relative weak topology of $B_X$.
    \end{enumequi}
  \item\label{item:b-reduce-lemma}
    The following assertions are equivalent:
    \begin{enumequi}
    \item\label{item:b-1-reduce}
      $X$ has property \CWOSX;
    \item\label{item:b-2-reduce}
      whenever $U_1$ and $U_2$ are
      relatively weakly open subsets of $B_X$,
      $\lambda_1, \lambda_2 > 0$,
      $\lambda_1 + \lambda_2 = 1$,
      and $x_1 \in U_1$, $x_2 \in U_2$
      are such that $\|\lambda_1 x_1 + \lambda_2 x_2\| = 1$,
      the element $\lambda_1 x_2 + \lambda_2 x_2$
      is an interior point of
      $\lambda_1 U_1 + \lambda_2 U_2$
      in the relative weak topology of $B_X$.
    \end{enumequi}
  \item\label{item:c-reduce-lemma}
    The following assertions are equivalent:
    \begin{enumequi}
    \item\label{item:c-1-reduce}
      $X$ has property \CWOBX;
    \item\label{item:c-2-reduce}
      whenever $U_1$ and $U_2$ are
      relatively weakly open subsets of $B_X$,
      $\lambda_1, \lambda_2 > 0$,
      $\lambda_1 + \lambda_2 = 1$,
      and $x_1 \in U_1$, $x_2 \in U_2$
      are such that $\|\lambda_1 x_1 + \lambda_2 x_2\| < 1$,
      the element $\lambda_1 x_2 + \lambda_2 x_2$
      is an interior point of
      $\lambda_1 U_1 + \lambda_2 U_2$
      in the relative weak topology of $B_X$.
    \end{enumequi}
  \end{enumerate}
\end{lem}

The proof of \ref{item:c-reduce-lemma},
\ref{item:c-2-reduce} $\Rightarrow$~\ref{item:c-1-reduce},
makes use of the following lemma.

\begin{lem}\label{lem: norm-less-than-one conv. comb.}
  Let $X$ be a Banach space, let $n \in \N$, let $U_1,\dotsc, U_n$ be
  relatively weakly open subsets of $B_X$,
  and let $\lambda_1,\dotsc, \lambda_n > 0$,
  $\sum_{j=1}^n \lambda_j = 1$.
  Then every $x \in \sum_{j=1}^n\lambda_j U_j$ with $\|x\| < 1$ can
  be written as
  \begin{equation*}
    x = \sum_{j=1}^n \lambda_j x_j,
    \quad
    \text{where}\ x_j \in U_j
    \ \text{and}\ \|x_j\| < 1
    \ \text{for every}\ j \in \{1,\dotsc,n\}.
     \end{equation*}
\end{lem}
\begin{proof}
  Let $x := \sum_{j=1}^n \lambda_j u_j$
  with $u_j \in U_j$, $j \in \{1,\dotsc,n\}$, be such that $\|x\| < 1$.
  Choosing $r \in (0,1)$ small enough,
  we have $x_j := r x + (1-r) u_j \in U_j$ and
  $\|x_j\| < 1$ for every $j \in\{1,\dotsc,n\}$.
  It remains to observe that
  \begin{equation*}
    \sum_{j=1}^n \lambda_j x_j
    = r x + (1-r) \sum_{j=1}^n \lambda_j u_j
    = r x + (1-r) x
    = x. \qedhere
  \end{equation*}
\end{proof}

\begin{proof}[Proof of Lemma \ref{lem: reducing to combinations of two}]
  In each of \ref{item:a-reduce-lemma}--\ref{item:c-reduce-lemma},
  the implication \ref{item:a-1-reduce} $\Rightarrow$~\ref{item:a-2-reduce}
  is trivial. \ref{item:a-2-reduce} $\Rightarrow$~\ref{item:a-1-reduce}
  follows easily by induction using the same idea of
  splitting the convex combination as in
  Proposition~\ref{prop: X has (co) <=> X has (co2)}.
  More precisely, for \ref{item:c-reduce-lemma},
  \ref{item:c-2-reduce} $\Rightarrow$~\ref{item:c-1-reduce},
  one first uses Lemma~\ref{lem: norm-less-than-one conv. comb.}.
\end{proof}

\begin{prop}\label{prop:norm1compl-inherit-cwox}
  Let $X$ be a Banach space. Norm-one complemented
  subspaces of $X$ inherit each of the properties
  \CWO, \CWOBX, and \CWOSX.
\end{prop}

\begin{proof}
  Let $Y$ be a subspace of $X$ and $P\colon X\to X$
  a projection onto $Y$ with $\|P\|=1$.
  Using Lemma~\ref{lem: reducing to combinations of two},
  it is enough to consider
  \[C_Y := \lambda_1U_1+\lambda_2U_2,\]
  where $U_1$ and $U_2$ are relatively weakly open subsets of $B_Y$
  and $\lambda_1,\lambda_2>0$ with $\lambda_1+\lambda_2=1$.
  Since $P$ is weak-to-weak continuous,
  $P^{-1}(U_1)\cap B_X$ and $P^{-1}(U_2)\cap B_X$
  are relatively weakly open in $B_X$. Set
  \[
    C_X := \lambda_1 \big(P^{-1}(U_1)\cap B_X\big)
    +\lambda_2 \big(P^{-1}(U_2)\cap B_X\big).
  \]
  Notice that $C_Y=C_X\cap B_Y$.
  This is immediate from $C_Y\subset C_X$ and $P(C_X)\subset C_Y$.

  If $X$ has property \CWO,
  then $C_X$ is relatively weakly open in $B_X$,
  therefore $C_Y$ is relatively weakly open in $B_Y$
  and it follows that $Y$ has property \CWO.

  For properties \CWOSX and \CWOBX,
  notice that $C_Y\cap S_Y=(C_X\cap S_X)\cap B_Y$
  and $C_Y\cap B_Y^\circ=(C_X\cap B_X^\circ)\cap B_Y$.
\end{proof}

In the case of \CWOBX, we can say a lot more.
Let $\lambda \ge 1$.
Recall that a closed subspace $Y$ of a Banach space $X$
is said to be \emph{locally $\lambda$-complemented} in $X$
if, for every finite dimensional subspace $E$ of $X$
and every $\varepsilon > 0$, there exists a linear
operator $P_E: E \to Y$ with $P_E x = x$ for all
$x \in E \cap Y$ and $\|P_E\| \le \lambda + \varepsilon$.
If $Y$ is locally $\lambda$-complemented in $X$, then there
exists an \emph{extension operator}
$\Phi: Y^* \to X^*$, that is $(\Phi y^*)(y) = y^*(y)$
for all $y \in Y$ and $y^* \in Y^*$.
Note that $\Phi$ can be chosen so that $\|\Phi\| \le \lambda$.
This was shown independently
by Fakhoury \cite[Th{\'e}or{\`e}me~2.14]{MR0348457}
and Kalton \cite[Theorem~3.5]{MR755269}.

\begin{thm}\label{thm: SD2P, CWO-B}
  Let $X$ be a Banach space with property \CWOBX
  and let $Z$ be an infinite dimensional Banach space such that,
  for every $\varepsilon > 0$, the space~$X$ contains
  a locally $(1+\varepsilon)$-complemented subspace
  which is $(1+\varepsilon)$-isometric to~$Z$.
  Then $Z$ has the strong diameter two property.
\end{thm}

\begin{proof}
  Suppose for contradiction that $Z$ fails the strong diameter two property.
  Then there are $n\in\N$ $\zs_1,\dotsc,\zs_n\in S_{Z^\ast}$, $\alpha>0$,
  $\lambda_1,\dotsc,\lambda_n>0$ with $\sum_{i=1}^n\lambda_i=1$, and $\rho\in(0,1)$
  such that, whenever $j\in\{1,2\}$, setting
  $\zs_{j,i}:=(-1)^{j-1}\zs_i$ and
  \begin{equation*}
    S^Z_{j,i} := S(\zs_{j,i},\alpha)
    = \{ z \in B_Z \colon \re \zs_{j,i}(z) > 1 - \alpha\},
  \end{equation*}
  one has $\diam(C^Z_j)< 2\rho$, where $C^Z_j:=\sum_{i=1}^n\lambda_i S^Z_{j,i}$.
  Observe that $\frac{1}{2}(C^Z_1 + C^Z_2) \subset B(0,\rho)$,
  because $S^Z_{2,i}=-S^Z_{1,i}$ and thus $C^Z_2=-C^Z_1$.

  Choose $\eps>0$ so that $1/{(1 + \eps)^2} > \max\{\rho,1-\alpha\}$,
  and let $Y$ be a closed locally $(1+\frac{\varepsilon}{2})$-complemented
  subspace of $X$ such that there exists
  an isomorphism $T \in \mathcal{L}(Y,Z)$ with $\|T\|\leq 1$
  and $\|T^{-1}\| < 1+\eps$.
  Since $Y$ is locally $(1+\frac{\varepsilon}{2})$-complemented, there exists
  an extension operator
  $\Phi : Y^* \to X^*$ with $\|\Phi\| \le 1 + \frac{\varepsilon}{2}$.
  Consider the slices
  \begin{equation*}
    S^X_{j,i} := \biggl\{ x \in B_X
    \colon \re(\Phi T^\ast\zs_{j,i})(x)
    >\frac{1}{1+\eps} \biggr\}
  \end{equation*}
  of $B_X$. To see that the sets $S^X_{j,i}$ are non-empty, observe that
  \begin{equation*}
    \|\Phi T^\ast\zs_{j,i}\|
    \ge \|\Phi T^\ast\zs_{j,i}|_Y\|
    =\|T^\ast\zs_{j,i}\|
    \geq\frac{\|\zs_{j,i}\|}{\|(T^\ast)^{-1}\|}
    = \frac{1}{\|(T^{-1})^\ast\|}
    > \frac{1}{1+\eps}.
  \end{equation*}
  Set $C^X_j:=\sum_{i=1}^n\lambda_i S^X_{j,i}$, $j\in\{1,2\}$.

  Since
  $0 \in \sum_{j=1}^2\sum_{i=1}^n\frac{\lambda_i}{2}S^X_{j,i}
  = \frac{1}{2}(C^X_1 + C^X_2)$,
  by courtesy of property \CWOBX{} for $X$,
  there is a relatively weakly open subset $W$ of $B_X$ such that
  $0 \in W \subset \frac{1}{2}(C^X_1 + C^X_2)$.
  Since $\widehat{W} := B_Y\cap W$ is non-empty
  (because $0 \in \widehat{W}$)
  and relatively weakly open in $B_Y$,
  there exists a $y \in S_Y \cap \widehat{W}$
  (here we use that $Y$ is infinite dimensional).
  Now $y\in\sum_{j=1}^2\sum_{i=1}^n\frac{\lambda_i}{2}S^X_{j,i}$,
  thus there are $x_{j,i}\in S^X_{j,i}$ such that
  $y=\sum_{j=1}^2\sum_{i=1}^n\frac{\lambda_i}{2}x_{j,i}$.
  Define
  \begin{equation*}
    E := \mbox{span}\{ x_{j,i} : j \in \{1,2\}, i \in
    \{1,\dotsc,n\} \} \subset X
  \end{equation*}
  and
  \begin{equation*}
    F := \mbox{span}\{ T^* z^*_{j,i} : j \in \{1,2\}, i \in
    \{1,\dotsc,n\} \} \subset Y^*.
  \end{equation*}
  Let $I_E : E \to X$ be the natural embedding.
  By \cite[Corollary~3.3]{OP3}, there exists
  a linear operator
  $P_E : E \to Y$ with $\|P_E\| \le 1 + \varepsilon$,
  $P_E x = x$ for all $x \in E \cap Y$, and
  $y^*(P_E x) = \Phi y^*(x)$ for all $x \in E$ and $y^* \in F$.
  Then $y = P_E y
  = \sum_{j=1}^2\sum_{i=1}^n\frac{\lambda_i}{2} P_E x_{j,i}$.
  For every $j\in\{1,2\}$ and every $i\in\{1,\dotsc,n\}$, one has
  \begin{equation*}
    y_{j,i} := \frac{1}{1+\eps} P_E x_{j,i}
    \in \biggl\{ y \in B_Y \colon
    \re(T^\ast\zs_{j,i})(y) > \frac{1}{(1 + \eps)^2}\biggr\}.
  \end{equation*}
  Set $y_j := \sum_{i=1}^n\lambda_i y_{j,i}$, $j\in\{1,2\}$.
  Observing that $Ty_{j,i}\in S^Z_{j,i}$,
  one has $Ty_j\in C^Z_j$, thus $v:=\frac{1}{2}(Ty_1+Ty_2)\in B(0,\rho)$ and
  $\frac{1}{2}(y_1+y_2)=T^{-1}v\in B\bigl(0,(1+\eps)\rho\bigr)$.
  It follows that
  \begin{equation*}
    1 = \|y\|
    = \biggl\|\sum_{j=1}^2\sum_{i=1}^n
    \frac{\lambda_i}{2} P_E x_{j,i} \biggr\|
    = (1 + \eps) \biggl\|\frac{1}{2}(y_1+y_2)\biggr\|
    \le \rho (1 + \eps)^2 < 1,
  \end{equation*}
  a contradiction.
\end{proof}

Using the fact that if a Banach space $X$ contains a complemented copy
of $\ell_1$, then it already contains, for any $\eps >0$, a
$(1+\eps)$-complemented subspace $(1 + \eps)$-isomorphic to $\ell_1$
\cite[Theorem~5]{MR1698052}, we immediately get

\begin{cor}\label{cor:compl-ell1-CWO-BX}
  If $X$ contains a complemented subspace isomorphic to $\ell_1$
  then $X$ does not have property \CWOBX.
\end{cor}

Since every nonreflexive subspace of an L-embedded space contains
a complemented subspace isomorphic to $\ell_1$
\cite[IV.Corollary~2.3]{MR1238713} we obtain the following.

\begin{cor}\label{cor:L1_fails_CWOBXo}
  Let $X$ be an L-embedded Banach space and $M$ a closed
  infinite dimensional subspace of $X$.
  Then $M$ does not have property \CWOBX.

  In particular,
  $L_1[0,1]$ does not have property \CWOBX.
\end{cor}

\section{The spaces $c_0(X_n)$ and $L_1(\mu)$}
\label{sec:5}

Let $\{X_n\}$ be a sequence of Banach spaces.
Then $c_0(X_n)$ is the Banach space of all norm null
sequences $(x_n)$, where $x_n \in X_n$ for all $n \in \N$,
with norm $\Vert (x_n)\Vert = \sup \{\Vert x_n\Vert :n\in \N \}$.
Note that the dual Banach
space of $c_0(X_n)$ is the Banach space $\ell_1(X_n^*)$
with norm
$\Vert (x_n^*)\Vert =\sum_{n=1}^{\infty}\Vert x_n^*\Vert$.

We will need a lemma similar to
Lemma~\ref{lem: C_0(K,X)* when K scattered -- betterment}.

\begin{lem}\label{lem:rwo-lemma-for-N-and-Xn}
  Let $\{X_n\}$ be a sequence of finite dimensional Banach spaces,
  and let $x \in B_{c_0(X_n)}$.
  \begin{enumerate}
  \item\label{item:1-c0X_n-lemma}
    Let $U$ be a neighbourhood of $x$
    in the relative weak topology of $B_{c_0(X_n)}$.
    Then there are a finite subset $M$ of $\N$
    and an $\varepsilon > 0$ such that, whenever $y \in B_{c_0(X_n)}$ satisfies
    \begin{equation}\label{eq: ||y(n)-x(n)||<eps for every n in M}
      \|y(m) - x(m)\| < \varepsilon
      \quad \text{for every}\ n \in M,
    \end{equation}
    one has $y \in U$.
  \item\label{item:2-c0X_n-lemma}
    Let $M$ be a finite subset of
    $\N$, and let $\eps > 0$.
    Then there is a neighbourhood $U$ of $x$ in the
    relative weak topology of $B_{c_0(X_n)}$ such that every $y \in U$
    satisfies \eqref{eq: ||y(n)-x(n)||<eps for every n in M}.
  \end{enumerate}
\end{lem}

\begin{proof}
  Set $Z := c_0(X_n)$. Then $Z^* = \ell_1(X_n^*)$.

  \ref{item:1-c0X_n-lemma}.
  Let a finite subset $\mathcal{F}$ of $S_Z$
  and an $\varepsilon > 0$ be such that
  \begin{equation*}
    \{ y  \in B_Z : |f(y) - f(x)| < 3\varepsilon
    \mbox{ for every } f \in \mathcal{F} \}
    \subset U.
  \end{equation*}
  For every $f = (x_n^*)_{n=1}^\infty \in \mathcal{F}$, let $f_N$
  be its projection onto the first $N$ coordinates,
  that is, $f_N = (z_n^*)$ where
  $z_n^* = x_n^*$ for $n \in\{1,\dotsc,N\}$ and $z_n^* = 0$
  for $n > N$. Choose $N$ so large that $\|f - f_N\| \le \varepsilon$
  for all $f \in \mathcal{F}$, and let $M = \{1,\dotsc,N\}$.
  Note that $\|f_N\| \le  \|f\| = 1$.

  Now, if $y\in B_Z$ satisfies \eqref{eq: ||y(n)-x(n)||<eps for every n in M},
  then, for every $f\in\mathcal{F}$,
  \begin{equation*}
    |f(y) - f(x)| \le
    |f(y) - f_N(y)| + |f_N(y) - f_N(x)| + |f_N(x) - f(x)|
    \le 3\varepsilon,
  \end{equation*}
  and it follows that $y \in U$.

  \ref{item:2-c0X_n-lemma}.
  For every $m \in M$, let $A_m \subset B_{X_m^*}$ be a finite
  $\frac{\varepsilon}{3}$-net for $B_{X_m^*}$.
  Set
  \begin{multline*}
    U := \biggl\{
        y \in B_{c_0(X_n)} :
        \text{$\bigl|x^*\bigl( y(m) - x(m) \bigr)\bigr| < \frac{\varepsilon}{3}$
        for every $m \in M$}\\
        \text{and every $x^* \in A_m$}
        \biggr\}.
  \end{multline*}
  Let $y \in U$ and $m \in M$ be arbitrary.
  Picking $z^* \in B_{X_m^*}$ so that
  \begin{equation*}
    z^*\bigl(y(m) - x(m)\bigr) = \|y(m) - x(m)\|,
  \end{equation*}
  there is an $x^* \in A_m$ satisfying
  $\|z^* - x^*\| < \frac{\varepsilon}{3}$.
  One has
  \begin{align*}
    \|y(m) - x(m)\|
    &= z^*\bigl(y(m) - x(m)\bigr) \\
    &\le \|z^* - x^*\|\|y(m) - x(m)\|
      + \bigl|x^*\bigl(y(m) - x(m)\bigr)\bigr| \\
    &\le \frac{2\varepsilon}{3} + \frac{\varepsilon}{3}
     = \varepsilon.
  \end{align*}
\end{proof}

\begin{thm}\label{thm:c0-sum}
  Let $\{X_n\}$ be a sequence of finite-dimensional
  Banach spaces with property \CCOM.
  Then $c_0(X_n)$ has property \CWO.
\end{thm}
\begin{proof}
  Set $Z := c_0(X_n)$.
  Let $V_1$ and $V_2$ be relatively weakly open subsets of $B_{Z}$
  and let $\lambda_1,\lambda_2 > 0$ with $\lambda_1+\lambda_2 = 1$.
  Using Lemma~\ref{lem: reducing to combinations of two},
  it is enough to
  consider the convex combination $C:=\lambda_1 V_1+\lambda_2 V_2$.
    Let $x =\lambda_1 x_1+\lambda_2 x_2\in C$ with $x_j \in V_j$.
  We are going to find a neighbourhood $U$ of $x$ in the
  relative weak topology of $B_Z$ such that $U \subset C$.

  By Lemma~\ref{lem:rwo-lemma-for-N-and-Xn}, \ref{item:1-c0X_n-lemma},
  there are an $\varepsilon > 0$ and a finite subset
  $M$ of $\N$ such that
  \begin{itemize}
  \item
    whenever $y_1,y_2 \in B_Z$ are such that, for every $m \in M$,
    \begin{equation*}
      \|y_j(m) - x_j(m)\| < \varepsilon,\quad j \in \{1,2\},
    \end{equation*}
    one has $y_j \in V_j$, $j \in \{1,2\}$.
  \end{itemize}
  For every $m \in M$, let $\delta_m$ and $\widehat{v}_{m,j}$,
  be, respectively, the $\delta$ and the functions
  $\widehat{v}_j$ from condition \CCOMn{} of Definition~\ref{def:P11}
  with $X=X_m$, $n=2$, $x = x(m)$, and $x_j = x_j(m)$, $j \in \{1,2\}$.
  By Lemma~\ref{lem:rwo-lemma-for-N-and-Xn}, \ref{item:2-c0X_n-lemma},
  there is a neighbourhood $U$ of $x$ in the relative weak topology
  of $B_Z$ such that, for every $y \in U$,
  \begin{equation*}
    \|y(m) - x(m)\| < \delta_m
    \quad \mbox{for every } m \in M.
  \end{equation*}

  Let $y\in U$ be arbitrary.   Define $y_1,y_2\in B_Z$ by
  $y_j(m) = \widehat{v}_{m,j}\bigl(y(m)\bigr)$ for every $m \in M$
  and $y_j(n) = y(n)$ for every $n\in\N \setminus M$.
  Then $y =\lambda_1 y_1+\lambda_2 y_2$ with $y_j \in V_j$,
  and hence $U \subset C$.
\end{proof}

It is known that the Banach space
$c_0(\ell_2^n)$ is not isomorphic to $c_0$
(here, by $c_0(\ell_2^n)$ we mean the space $c_0(X_n)$,
where $X_n=\ell_2^n$ for every $n\in\N$).
By the above theorem, and
Proposition~\ref{prop:extB_X_and_points_in_B^0_X_have_P11},
\ref{item:prop_P11-2},
the space $c_0(\ell_2^n)$ has property \CWO.
In fact, we have the following result.

\begin{cor}\label{cor:c_0ofell_pn_cwo}
  Whenever $1 \le p \le \infty$, the Banach space $c_0(\ell_p^n)$
  has property \CWO.
\end{cor}

\begin{proof}
  We use Theorem \ref{thm:c0-sum} together
  with Theorem~\ref{thm:complex-ell_1^n} if $p=1$,
   with Proposition~\ref{prop:extB_X_and_points_in_B^0_X_have_P11},
  \ref{item:prop_P11-2}, if $1 < p < \infty$, and with
  Proposition~\ref{prop:polyhedral-dual-ccom} if $p = \infty$.
\end{proof}

The above result does not hold for $\ell_\infty$-sums.
Not even $\ell_\infty = C(\beta \N)$ has property \CWO
since $\beta \N$ is not scattered
(that would force $\ell_\infty$ to be Asplund
\cite[Theorem~14.25]{MR2766381}), and $C(K)$
has property \CWO only when $K$ is a
scattered compact Hausdorff space
\cite[Theorem~3.1]{2017arXiv170302919H}.

In \cite[Remark~IV.5]{MR912637}, it was observed
that finite convex combinations of relatively
weakly open subsets of the positive face
of the unit ball of $L_1[0,1]$ are still relatively weakly open.
Our next theorem shows that the same (and even more) holds
for the whole unit sphere.
\begin{rem}\label{rem:L1mu-atomless-sphere}
  Let $\mu$ be a non-zero $\sigma$-finite (countably additive non-negative)
  measure on a $\sigma$-algebra $\Sigma$
  of a non-empty set $\Omega$.
  Then $\mu$ is atomless if and only if
  every finite convex combination
  of relatively weakly open subsets of $B_{L_1(\mu)}$
  intersects the unit sphere.

  Indeed, let $U_1, \dotsc, U_n$ be non-empty
  relatively weakly open subsets of $B_{L_1(\mu)}$.
  By Bourgain's lemma \cite[Lemma~II.1]{MR912637},
  each $U_j$ contains a finite convex combination
  of slices (of $B_{L_1(\mu)}$). Hence any convex combination of
  the $U_j$'s contains a finite convex combination of slices.
  Identifying $L_1(\mu)^*$ with $L_\infty(\mu)$,
  one can easily show that if $\mu$ is atomless,
  then any convex combination
  of slices of $B_{L_1(\mu)}$---and thus also any convex combination
  of the $U_j$'s---intersects $S_X$
  (see \cite[Example~3.2]{2017arXiv170106169A} for an argument).

  On the other hand, if $A\in\Sigma$ is an atom for $\mu$,
  then $\frac{1}{\mu(A)}\chi_A\in B_{L_1(\mu)}$ is strongly exposed by
  $g:=\chi_A\in B_{L_\infty(\mu)}$
  (here we identify  $L_1(\mu)^*$ with $L_\infty(\mu)$ again).
  It follows that if $\alpha>0$
  is small enough, then the slices $S_1:=S(g,\alpha)$ and $S_2:=S(-g,\alpha)$
  of $B_{L_1(\mu)}$ have diameter less than $1$. Now, the convex combination
  $C:=\frac{1}{2}S_1+\frac{1}{2}S_2$ contains~$0$
  and has diameter less than $1$; thus $C$ does not intersect the unit sphere.
\end{rem}
As a consequence of the results in \cite{MR2326380}
and the above remark, $L_1(\mu)$ has the Daugavet property
if and only if every finite convex combination
of relatively weakly open sets of its unit ball intersects the unit sphere.
Next we show that every point in such an intersection is an interior point
of the corresponding convex combination
in the relative weak topology of $B_{L_1(\mu)}$.

\begin{thm}\label{thm:L1mu-CWOSX}
  Let $\mu$ be a non-zero $\sigma$-finite
  (countably additive non-negative)
  measure on a sigma-algebra $\Sigma$ of a non-empty set $\Omega$.
  Then the real space $L_1(\mu)$ has property \CWOSX.
\end{thm}
\begin{proof}
  Let $U_1$ and $U_2$ be relatively weakly open subsets of the closed
  unit ball $B_{L_1(\mu)}$ of $L_1(\mu)$,
  and let $\lambda_1, \lambda_2>0$, $\lambda_1 + \lambda_2 = 1$,
  and $x_1 \in U_1$, $x_2 \in U_2$ be
  such that $\|\lambda_1 x_1 + \lambda_2 x_2 \|=1$.
  By Lemma~\ref{lem: reducing to combinations of two}
  it is enough to find a neighbourhood $W$ of
  $x := \lambda_1 x_1 + \lambda_2 x_2$ in
  the relative weak topology of $B_{L_1(\mu)}$
  such that $W\subset \lambda_1 U_1 + \lambda_2 U_2$.

  Throughout the proof, whenever convenient, we identify functionals
  in $L_1(\mu)^\ast$ with elements in $L_\infty(\mu)$ in the canonical
  way.
  Since $L_\infty(\mu)=\overline{\spann}\{\chi_E\colon\,E\in\Sigma\}$,
  there are a finite collection $\mathcal{F}$ of subsets of $\Sigma$
  and an $\eps>0$ such that
  \begin{equation*}
    V_i := \biggl\{u \in B_{L_1(\mu)} :
    \biggl |\int_{E} (u-x_i) \,d\mu \biggr|
    < 2\eps
    \ \text{for every }
    E \in \mathcal{F} \biggr\}
    \subset U_i,\quad i \in \{1,2\}.
  \end{equation*}
   We may assume that $\bigcup_{E\in\mathcal{F}}E=\Omega$, that
  the sets in $\mathcal{F}$ are pairwise disjoint,
  and that, for every $E \in \mathcal{F}$,
  either $x_1\chi_E \geq 0$ a.e. and $x_2\chi_E \geq 0$ a.e., or
  $x_1\chi_E\leq 0$ a.e. and $x_2 \chi_E \leq 0$ a.e.
  (the latter is because if $x_1x_2(t) < 0$ for almost every $t$ in a
  set $D\in\Sigma$ with $\mu(D) > 0$, then one would have
  $\|x\| = \|\lambda_1 x_1 + \lambda_2 x_2\| < 1$).

  Set $E_0:=\bigcup_{E\in\mathcal{F}_0}E$ where
  $\mathcal{F}_0:=\{E \in \mathcal{F} : \int_{E} x \,d\mu=0\}$,
  and label the sets in $\mathcal{F}\setminus\mathcal{F}_0$ as
  $E_1,\dotsc,E_n$ where $n\in\N$.
  For every $i\in \{1,\dotsc,n\}$, set
  \begin{equation*}
    \alpha_i := \int_{E_i} x_1 \,d\mu,
    \quad
    \beta_i := \int_{E_i} x_2 \,d\mu,
    \quad\text{and}\quad
    \gamma_i := \lambda_1\alpha_i + \lambda_2 \beta_i
    = \int_{E_i} x \,d\mu.
  \end{equation*}
  Note that, since $x_1$ and $x_2$ have the same sign
  on each $E_i$, we have $\frac{\alpha_i}{\gamma_i} \ge 0$,
  $\frac{\beta_i}{\gamma_i} \ge 0$, and
  $|\gamma_i| = \int_{E_i} |x| \,d\mu$
  for every $i \in \{1,\dotsc,n\}$.

  Set $J_0:=\bigl\{i\in \{1,\dotsc,n\}\colon\,
  \text{$\alpha_i=0$ or $\beta_i=0$}\bigr\}$,
  $J_1:=\{1,\dotsc,n\}\setminus J_0$,
  and, if $J_1\ne\emptyset$, then set
  $\Gamma :=\sum_{i\in J_1}|\gamma_i|$.
  Pick $\delta>0$ so that $Mn^2\delta<\eps$ where
  $M:=\max_{1\leq i\leq n}\frac{\alpha_i+\beta_i}{\gamma_i}$.
  We may assume that if $J_1\ne\emptyset$, then
  $\delta<\frac{|\gamma_i|}{2}$
  and
  $\frac{2Mn^2\delta}{\Gamma}<\min\bigl\{\frac{\alpha_i}{\gamma_i},
  \frac{\beta_i}{\gamma_i}\bigr\}$ for every $i\in J_1$.

  Define
  \begin{equation*}
    W:=\biggl\{w\in
    B_{L_1(\mu)}\colon\,\text{$\biggl|\int_{E_i}w\,d\mu-\gamma_i\biggr|<\delta$
      for every $i\in \{1,\dotsc,n\}$}\biggr\}.
  \end{equation*}
  Let $w\in W$ be arbitrary. For every $i\in\{0,1,\dotsc,n\}$, set
  $w_i=w\chi_{E_i}$ and
  $\eta_i:=\|w_i\|-|\gamma_i|$.
  Then
  \begin{align*}
    |\gamma_i|+\eta_i=\|w_i\|=\int_{E_i}|w|\,d\mu \geq
    \Bigl|\int_{E_i}w\,d\mu\Bigr|>|\gamma_i|-\delta,
  \end{align*}
  whence $\eta_i > -\delta$. On the other hand,
  \begin{align*}
    1 \geq \|w\| = \sum_{j=0}^n\|w_j\| \geq
    \sum_{j=1}^n|\gamma_j| + \eta_i + \sum_{\substack{j=1\\j\ne i}}^n
    \eta_j
    > 1 + \eta_i-(n-1)\delta,
  \end{align*}
  whence $\eta_i < (n-1)\delta$, and thus $|\eta_i| < n\delta$.
  Observe that if $J_1\ne\emptyset$, then
  \begin{align*}
    \rho: = \sum_{i\in J_1}\|w_i\|
    = \sum_{i\in J_1}\bigl(|\gamma_i|+\eta_i\bigr)
    > \sum_{i\in J_1}\bigl(|\gamma_i|-\delta)
    > \sum_{i\in J_1}\frac{|\gamma_i|}{2}
    =\frac{\Gamma}{2}.
  \end{align*}
  Setting
  \[
    c := \sum_{i=1}^n\frac{\beta_i-\alpha_i}{\gamma_i}\,\eta_i,
  \]
  one has
  \[
  |c|\leq M\sum_{i=1}^n|\eta_i|<Mn^2\delta,
  \]
  and thus, for every $i\in J_1$,
  \begin{align*}
    \frac{|c|}{\rho}
    \leq \frac{2|c|}{\Gamma}
    < \frac{2Mn^2\delta}{\Gamma}
    < \min\biggl\{\frac{\alpha_i}{\gamma_i},\frac{\beta_i}{\gamma_i}\biggr\}.
  \end{align*}

  Define $u_0 := v_0 := w_0$ and,
  for every $i\in \{1,\dotsc,n\}$,
  \begin{equation*}
    \left\{
      \begin{aligned}
        u_i&:= \frac{\alpha_i}{\gamma_i} w_i,\\
        v_i&:= \frac{\beta_i}{\gamma_i} w_i
      \end{aligned}
    \right.
    \quad\text{if $i\in J_0$,}
    \qquad\text{and}\qquad
    \left\{
      \begin{aligned}
        u_i&:=
        \left(\frac{\alpha_i}{\gamma_i} + \lambda_2\,\frac{c}{\rho}\right) w_i, \\
        v_i&:=
        \left(\frac{\beta_i}{\gamma_i} - \lambda_1\,\frac{c}{\rho}\right) w_i
      \end{aligned}
    \right.
    \quad\text{if $i\in J_1$.}
  \end{equation*}
  Define $u := \sum_{i=0}^n u_i$ and
  $v := \sum_{i=0}^n v_i$.
  For every $i\in \{0,1,\dotsc,n\}$,
  one has $\lambda_1 u_i + \lambda_2 v_i = w_i$;
  thus $\lambda_1 u + \lambda_2 v = \sum_{i=0}^n w_i = w$.
  Now,
  \begin{align*}
    \|u\| &=
            \|w_0\|
            + \sum_{i \in J_0} \frac{\alpha_i}{\gamma_i} \|w_i\|
            + \sum_{i \in J_1} \left(\frac{\alpha_i}{\gamma_i}
            + \lambda_2\,\frac{c}{\rho}\right) \|w_i\| \\
          &=
            \|w_0\|
            + \sum_{i=1}^n \frac{\alpha_i}{\gamma_i} \|w_i\|
            + \lambda_2 c \\
          &=
            \|w_0\|
            +
            \sum_{i=1}^n \frac{\alpha_i}{\gamma_i}
            (|\gamma_i| + \eta_i)
            + \lambda_2 \sum_{i=1}^n\frac{\beta_i-\alpha_i}{\gamma_i}\,\eta_i \\
          &=
            \|w_0\|
            +
            \sum_{i=1}^n |\alpha_i|
            + \sum_{i=1}^n\frac{\lambda_1\alpha_i+\lambda_2\beta_i}{\gamma_i}\,\eta_i \\
          &=
            \|w_0\| + 1 + \sum_{i=1}^n\eta_i,
  \end{align*}
  and, similarly, $\|v\|=\|w_0\| + 1 + \sum_{i=1}^n\eta_i$.
      Since
  \begin{equation}\label{eq: ||w_0||=<...}
   \|w_0\|
    =\|w\| - \sum_{i=1}^n\|w_i\|
    \leq 1 - \sum_{i=1}^n|\gamma_i| - \sum_{i=1}^n\eta_i
    = -\sum_{i=1}^n\eta_i,
  \end{equation}
   it follows that $\|u\|, \|v\| \le 1$.

  It remains to show that $u\in U_1$ and $v\in U_2$.
  To this end, first observe that, for every $i\in \{1,\dotsc,n\}$,
  \begin{equation*}
    \biggl|\frac{\alpha_i}{\gamma_i}\int_{E_i}w\,d\mu - \alpha_i\biggr|
    =\frac{\alpha_i}{\gamma_i}\,
    \biggl|\int_{E_i}w\,d\mu - \gamma_i \biggr|
    < M\delta < \eps.
  \end{equation*}
  Thus $\bigl| \int_{E_i}u\,d\mu-\alpha_i \bigr| < \eps$
  for every $i\in J_0$.
  For every $i\in J_1$,
  \begin{equation*}
    \biggl| \int_{E_i}\frac{c}{\rho}w \,d\mu \biggr|
    \leq \frac{|c|}{\rho} \int_{E_i}|w| \,d\mu
    = \frac{|c|}{\rho} \|w_i\|
    \leq |c|
    < Mn^2\delta
    < \eps,
  \end{equation*}
    thus
  \begin{equation*}
    \biggl| \int_{E_i}u\,d\mu - \alpha_i \biggr|
    \leq\biggl| \frac{\alpha_i}{\gamma_i}
    \int_{E_i} w \,d\mu - \alpha_i \biggr|
    + \lambda_2\biggl| \int_{E_i}\frac{c}{\rho} w \,d\mu \biggr|
    < 2\eps.
  \end{equation*}
   Since, by \eqref{eq: ||w_0||=<...},
  $\|w_0\|\leq- \sum_{i=1}^n\eta_i<n\delta<\eps$,
  one has, for every $E \in \mathcal{F}_0$,
  \begin{equation*}
    \biggl| \int_E u\,d\mu - \int_E x_1 \,d\mu \biggr|
    =\biggl| \int_E w_0\,d\mu \biggr|
      \leq \int_{\Omega} |w_0|\,d\mu
      = \|w_0\| < \eps.
  \end{equation*}
  It follows that $u \in V_1 \subset U_1$.
  One can similarly show that $v \in V_2 \subset U_2$,
  and the proof is complete.
\end{proof}

\section{Questions}
\label{sec:quest}

We end the paper with some natural questions.

\begin{enumerate}
\item
  All of our infinite dimensional examples
  of spaces with property \CWO contain $c_0$,
  that is, both $C_0(K,X)$ and $c_0(X_n)$ contain
  a copy of $c_0$.
  Must every Banach space with property \CWO
  contain a copy of $c_0$?
\item
  Does there exist a dual (infinite dimensional)
  Banach space with property \CWO?
\item
  If both $X$ and $Y$ have property \CWO,
  does the injective tensor product
  $X \widehat{\otimes}_\eps Y$ also have property \CWO?

  We could also ask the same if both
  $X$ and $Y$ have either property \CWOSX or \CWOBX.
\end{enumerate}

\bibliographystyle{amsplain}
\providecommand{\bysame}{\leavevmode\hbox to3em{\hrulefill}\thinspace}
\providecommand{\MR}{\relax\ifhmode\unskip\space\fi MR }
\providecommand{\MRhref}[2]{%
  \href{http://www.ams.org/mathscinet-getitem?mr=#1}{#2}
}
\providecommand{\href}[2]{#2}

\end{document}